\documentclass[10pt]{amsart}
\usepackage{amsmath,amssymb}
\usepackage{enumerate}
\usepackage{color}
\usepackage{hyperref}
\usepackage{natbib}

\usepackage{graphicx}
\usepackage{dsfont}

\newtheorem{theorem}{Theorem}[section]

\newtheorem{lemma}[theorem]{Lemma}
\newtheorem{corollary}[theorem]{Corollary}
\newtheorem{definition}[theorem]{Definition}
\newtheorem{remark}[theorem]{Remark}
\newtheorem{example}[theorem]{Example}
\newtheorem{assumption}[theorem]{Assumption}

\newcommand{\R}{\mathbb{R}}

\newcommand{\N}{\mathbb{N}}
\newcommand{\E}{\mathbb{E}}

\renewcommand{\P}{\mathds{P}}

\renewcommand{\epsilon}{\varepsilon}

\usepackage{ifthen} 
\usepackage{xifthen} 
\usepackage{twoopt}

\newcommandtwoopt\picausal[2][][] {
\ifthenelse{\isempty{#1}}
	{\ifthenelse{\isempty{#2}} {\widehat{\boldsymbol{\pi}}^N} {\widehat{\boldsymbol{\pi}}^{N,{#2}}}}
	{\ifthenelse{\isempty{#2}} {\widehat{\boldsymbol{\pi}}^N_{{#1}}} {\widehat{\boldsymbol{\pi}}^{N, {#2}}_{{,#1}}}}
}

\newcommandtwoopt\nucausal[2][][] {
\ifthenelse{\isempty{#1}}
	{\ifthenelse{\isempty{#2}} {\widehat{\boldsymbol{\nu}}^N} {\widehat{\boldsymbol{\nu}}^{N,{#2}}}}
	{\ifthenelse{\isempty{#2}} {\widehat{\boldsymbol{\nu}}^N_{{#1}}} {\widehat{\boldsymbol{\nu}}^{N, {#2}}_{{,#1}}}}
}

\newcommand\causal[2][]{
\ifthenelse{\isempty{#2}}  {{#1}^N_{\square} }
{{#1}^N_{\square {,#2}}} }

\newcommandtwoopt\muempirical[2][][]{
\ifthenelse{\isempty{#1}}
	{\ifthenelse{\isempty{#2}} {\widehat{\mu}^N} {\hat{\mu}^{N,{#2}}}}
	{\ifthenelse{\isempty{#2}} {\widehat{\mu}^N_{{#1}}} {\widehat{\mu}^{N, 					{#2}}_{{#1}}}}
}

\begin{document}

\title[A Wasserstein estimator for measuring association]{Measuring association with Wasserstein distances}

\date{\today}

\author{Johannes Wiesel}
\address{Johannes Wiesel\newline
Columbia University, Department of Statistics\newline
1255 Amsterdam Avenue\newline
New York, NY 10027, USA}
\email{johannes.wiesel@columbia.edu}


\begin{abstract}
\ Let $\pi\in \Pi(\mu,\nu)$ be a coupling between two probability measures $\mu$ and $\nu$ on a Polish space. In this article we propose and study a class of nonparametric measures of association between $\mu$ and $\nu$, which we call Wasserstein correlation coefficients. These coefficients are based on the Wasserstein distance between $\nu$ and the disintegration $\pi_{x_1}$ of $\pi$ with respect to the first coordinate. We also establish basic statistical properties of this new class of measures: we develop a statistical theory for strongly consistent estimators and determine their convergence rate in the case of compactly supported measures $\mu$ and $\nu$. Throughout our analysis we make use of the so-called adapted/bicausal Wasserstein distance, in particular we rely on results established in [Backhoff, Bartl, Beiglb\"ock, Wiesel. Estimating processes in adapted Wasserstein distance. 2020]. Our approach applies to probability laws on general Polish spaces. 
\end{abstract}

\thanks{ We thank Bodhi Sen for helpful discussions.}

\maketitle

\section{Introduction}

Given a sample of $(X_1^1, X_2^1), (X_1^2, X_2^2), \dots, (X_1^N, X_2^N) $ generated from a probability measure $\pi$ with marginals $\mu$ and $\nu$ on a product $\mathcal{X}\times \mathcal{Y}$ of topological spaces, a number of works have recently asked whether it is possible to define a simple empirical measure $T_N$, which provides an estimate for a non-parametric measure of association between $\mu$ and $\nu$. More concretely, \cite[Abstract]{chatterjee2020new} states the following desirable conditions:

\begin{quote}
``Is it possible to define a coefficient of correlation which is:
\begin{enumerate}[(i)]
\item simple as the classical coefficients like Pearson's correlation or Spearman's correlation, and yet
\item  Consistently estimates some simple and interpretable measure of the degree of dependence between the variables, which is 0 if and only if the variables are independent and 1 if and only if one is a measurable function of the other, and
\item  Has a simple asymptotic theory under the hypothesis of independence, like the classical coefficients?"
\end{enumerate}
\end{quote}

As is argued in \cite{chatterjee2020new}, none of the various past works based on joint cumulative distribution functions and ranks, kernel-based methods, information theoretic coefficients, coefficients based on copulas or on pairwise distances (see e.g. \cite{renyi1959measures,linfoot1957informational, blum1961distribution, rosenblatt1975quadratic,schweizer1981nonparametric, friedman1983graph, scarsini1984measures, szekely2007measuring,lyons2013distance,  gamboa2018sensitivity, zhang2019bet, puccetti2019measuring} and the references therein) satisfy all three properties stated above. It turns out that the articles \cite{trutschnig2011strong, junker2021estimating} and later also \cite{dette2013copula} and \cite{chatterjee2020new} are the first to answer Chatterjee's question above in the affirmative for spaces $\mathcal{X}=\R^{d_1}$ and $\mathcal{Y}=\R^{d_2}$, where $d_2=1$.  Since then their correlation coefficient has attracted a lot of attention, see e.g. \cite{shi2020power,cao2020correlations} and \cite{griessenberger2021multivariate} for a very recent extension. Complementary to this approach, \cite{deb2020kernel} (see also \cite{ke2019expected}) show how to build a corresponding estimator $T_N$ for general $d_2\ge 1$. The analysis in \cite{deb2020kernel} is restricted to estimators arising from reproducing kernel Hilbert spaces with specific requirements on the kernel and thus cannot be applied to arbitrary Polish spaces $\mathcal{X}$, $\mathcal{Y}$. In this article we offer an alternative construction of $T_N$ based on adapted Wasserstein distances. The idea of using tools from optimal transport to measure dependence of measures is not new and can be traced back at least to \cite{gini1915nuovi}, see also \cite{cifarelli2017centennial}. Recently, this subject has seen a spike in research activity: current works comprise \cite{ozair2019wasserstein,xiao2019disentangled, mori2020earth, mordant2021measuring, nies2021transport} amongst others. However, to the best of our knowledge this article is the first to define a coefficient of correlation based on adapted Wasserstein distances for general Polish spaces and to derive its analytical and statistical properties from adapted optimal transport. Indeed, directly utilising the underlying compatible metric structure of the space $\mathcal{X}$, the above properties (i)-(iii) hold without further assumptions. Furthermore, by varying the metric $d$ and the Wasserstein exponent $p$, one can naturally construct a whole family of different Wasserstein correlation coefficients, while directly exploiting the theory of optimal transportation. In fact, it will turn out that once we have defined the Wasserstein correlation $\overrightarrow{\mathcal{W}}$, our estimators can be computed via the plug-in approach $\overrightarrow{\mathcal{W}}_N=\overrightarrow{\mathcal{W}}(\picausal)$ for the so-called adapted empirical measure $\picausal$. In this article we derive consistency and convergence rates of the estimator $\overrightarrow{\mathcal{W}}(\picausal)$ under different assumptions.

\section{Notation and main results}\label{sec:1}

Let $\mathcal{X}$ be a Polish space with a compatible metric $d$ and let us denote by $\mathrm{Prob}(\mathcal{X})$ the set of Borel probability measures on $\mathcal{X}$. Let us take $\mu,\nu\in \text{Prob}(\mathcal{X})$ and denote by $\Pi(\mu,\nu)$ the set of couplings between $\mu$ and $\nu$ as , i.e. 
\begin{align*}
\Pi(\mu,\nu)= \left\{\pi \in \mathrm{Prob}(\mathcal{X}\times \mathcal{X}) :\ \pi(\cdot\times \mathcal{X})=\mu(\cdot), \ \pi(\mathcal{X}\times \cdot)=\nu(\cdot)\right\}.
\end{align*}
The Wasserstein distance $\mathcal{W}(\mu,\nu)$ is defined via
\begin{align*}
\mathcal{W}(\mu,\nu)=\inf_{\pi \in \Pi(\mu,\nu)} \int d(x_1,x_2)\,\pi(dx_1,dx_2).
\end{align*}
The pushforward of the measure $\mu$ via a function $f:\mathcal{X}\to \mathcal{X}$ is denoted by $f_\#\mu$, i.e. $$(f_\#\mu)(A):=\mu(\{x\in \mathcal{X}:\ f(x)\in A\})$$ for all Borel sets $A\subseteq \mathcal{X}$. Generalising the above definition to Borel probability measures on $\mathcal{X}^2:=\mathcal{X}\times\mathcal{X}$, we often write $\pi_1=(x_1)_\#\pi$ and $\pi_2=(x_2)_\#\pi$  for $\pi\in \text{Prob}(\mathcal{X}^2)$, where $(x_1,x_2)\mapsto x_1$ and $(x_1,x_2)\mapsto x_2$ are the canonical projection maps from $\mathcal{X}^2$ to the first and second coordinates respectively. We also recall that any coupling $\pi\in \Pi(\mu,\nu)$ has a $\mu$-a.s. unique disintegration with respect to the first coordinate, i.e. there exists a Borel measurable function $x_1\mapsto \pi_{x_1}$ such that $$\pi(A\times B)=\int_A\pi_{x_1}(B)\,\mu(dx_1) \quad \text{for all Borel sets }A,B\subseteq \mathcal{X}.$$ The product coupling with marginals $\mu$ and $\nu$ is denoted by $\mu\otimes \nu$.\\
One of the key notions used in this article is the so-called adapted Wasserstein distance. It can be introduced as follows:
for Borel probability measures $\pi,\tilde{\pi}$ on $\mathcal{X}^2$ we define the adapted (sometimes also called nested or bicausal) Wasserstein distance $\mathcal{AW}(\pi,\tilde{\pi})$ via
\begin{align}\label{eq:causal}
\mathcal{AW}(\pi,\tilde{\pi})=\inf_{\gamma\in \Pi(\pi_1, \tilde{\pi}_1)} \int \left[d(x_1, y_1)+\mathcal{W}(\pi_{x_1}, \tilde{\pi}_{y_1})\right]\,\gamma(dx_1,dy_1).
\end{align} 
On an intuitive level, the nested Wasserstein distance only considers those couplings $\gamma\in \Pi(\pi, \tilde{\pi})$, which respect the information flow formalised by the canonical (i.e. coordinate) filtration $(\mathcal{F}_t)_{t\in \{1,2\}}$: in \eqref{eq:causal} this is achieved by first taking an infimum over  couplings of $\pi_1, \tilde{\pi}_1$ (i.e. ``couplings at time one") and then a second (nested) infimum with respect to the respective disintegrations (i.e. ``conditional couplings at time two"). This feature distinguishes $\mathcal{AW}$ from the Wasserstein distance $\mathcal{W}$, which also includes ``anticipative couplings". We refer to \cite[pp. 2-3]{BackhoffVeraguas:2019tnb} for a well-written introduction to this topic. The nested distance was introduced in \cite{Pflug:2009hl}, \cite{Pflug:2012bfa} in the context of multistage stochastic optimisation and was independently analysed in \cite{Lassalle:2018hfa}.\
Let us also remark here that we always have the inequality
\begin{align}\label{eq:AW_W}
\mathcal{W}(\pi,\tilde{\pi})\le \mathcal{AW}(\pi,\tilde{\pi}),
\end{align}
where the Wasserstein distance $\mathcal{W}(\pi,\tilde{\pi})$ is correspondingly defined as
\begin{align*}
\mathcal{W}(\pi,\tilde{\pi}) =\inf_{\gamma\in \Pi(\pi,\tilde{\pi})} \int \left[d(x_1,y_1)+d(x_2,y_2)\right]\,\gamma(d(x_1, x_2), d(y_1,y_2))
\end{align*}
and 
\begin{align*}
\Pi(\pi,\tilde{\pi})=\left\{\gamma\in \Pi(\pi,\tilde{\pi}):\ \gamma(\cdot \times \mathcal{X}^2)=\pi(\cdot), \ \gamma (\mathcal{X}^2\times \cdot)=\tilde{\pi}(\cdot)\right\}.
\end{align*}

For the rest of this article we fix two measures $\mu,\nu\in \text{Prob}(\mathcal{X})$. We now introduce the following measure of association, which will be the main concept discussed in this article:

\begin{definition}
For any $\pi\in \text{Prob}(\mathcal{X}^2)$ we define the \emph{Wasserstein correlation coefficient} $\pi\mapsto \overrightarrow{\mathcal{W}}(\pi)$ by
\begin{align*}
\overrightarrow{\mathcal{W}}(\pi):=\frac{\int \mathcal{W}(\pi_{x_1},\pi_2)\,\pi_1(dx_1)}{\int  d(y,z)\,\pi_2(dy)\,\pi_2(dz)}.
\end{align*}
If $\pi\in \Pi(\mu,\nu)$, then  in particular
\begin{align*}
\overrightarrow{\mathcal{W}}(\pi)=\frac{\int \mathcal{W}(\pi_{x_1},\nu)\,\mu(dx_1)}{\int  d(y,z)\,\nu(dy)\,\nu(dz)},
\end{align*}
where throughout we assume that $\nu$ is not a singleton, i.e. $$\int  d(y,z)\,\nu(dy)\nu(dz)\neq 0.$$ 
\end{definition}

The key idea for the definition of $\overrightarrow{\mathcal{W}}(\pi)$ is the following insight: in order to capture association between $\mu$ and $\nu$ for a coupling $\pi\in \Pi(\mu,\nu)$, it is sufficient to compare the disintegration $\pi_{x_1}$ with $\nu$ via the term $\mathcal{W}(\pi_{x_1},\nu)$. Indeed, this term is zero for $\mu$-a.e. $x_1$ if and only if $\pi_{x_1}=\nu$ or equivalently $\pi=\mu\otimes\nu$. On the other hand, if $\mu$ and $\nu$ are completely dependent then $\pi_{x_1}=\delta_{f(x_1)}$ for some function $f$. In this case there is only one coupling between $\pi_{x_1}$ and $\nu$ and the value of the denominator and the numerator in the above definition align, yielding $\overrightarrow{\mathcal{W}}(\pi)=1$.\\
More generally, defining $\pi^\lambda=\lambda\pi^1+(1-\lambda)\pi^2$ for some measures $\pi^1, \pi^2\in \Pi(\mu,\nu)$ and $\lambda\in [0,1]$, we conclude by convexity of Wasserstein distances  that $\overrightarrow{\mathcal{W}}(\pi^\lambda)\le \lambda \overrightarrow{\mathcal{W}}(\pi^1)+(1-\lambda)\overrightarrow{\mathcal{W}}(\pi^2)$, i.e. $\pi\mapsto \overrightarrow{\mathcal{W}}(\pi)$ is convex on $\Pi(\mu,\nu)$. We discuss further properties of the Wasserstein correlation in the upcoming sections. In particular we show that $\overrightarrow{\mathcal{W}}$ indeed satisfies the main requirement (ii) stated in \cite[Abstract]{chatterjee2020new}, as cited in the introduction:

\begin{theorem}\label{thm:easy}
For any $\pi\in \Pi(\mu,\nu)$ the functional $\pi\mapsto \overrightarrow{\mathcal{W}}(\pi)$ satisfies:
\begin{enumerate}[(i)]
\item $\overrightarrow{\mathcal{W}}(\pi)\in [0,1]$.
\item $\overrightarrow{\mathcal{W}}(\pi)=0$ if and only if $\pi=\mu\otimes\nu$.
\item  $\overrightarrow{\mathcal{W}}(\pi)=1$ if and only if $\nu=f_\#\mu$ for some measurable function $f:\mathcal{X}\to \mathcal{X}$.
\end{enumerate}
\end{theorem}

A natural estimator for $\overrightarrow{\mathcal{W}}$ is given via the following plugin approach:
\begin{theorem}
Let $\pi \in \Pi(\mu,\nu)$ such that 
\begin{align*}
\int d(x_2,x_0)\,\nu(dx_2) <\infty
\end{align*}
for any $x_0\in \mathcal{X}$ and let $\hat{\pi}^N$ be an $\mathcal{AW}$-consistent estimator of $\pi$. Then $\overrightarrow{\mathcal{W}}(\hat{\pi}^N)$ is a consistent estimator of $\overrightarrow{\mathcal{W}}(\pi)$.
\end{theorem}

One such $\mathcal{AW}$-consistent estimator of $\pi$ has recently been constructed in \cite{backhoff2020estimating} and throughout this article, we will make use of results established there. In particular continuity of $\overrightarrow{\mathcal{W}}$ in $\mathcal{AW}$ will directly enable us to establish convergence rates for $\overrightarrow{\mathcal{W}}(\picausal)$.\\

Let us also also remark that our analysis can easily be extended to consider $p$-Wasserstein distances $\mathcal{W}_p$ for $p>1$ by correspondingly considering the $p$-Wasserstein correlation coefficient
$$ \overrightarrow{\mathcal{W}}_p(\pi):=\frac{\left(\int \mathcal{W}_p(\pi_{x_1},\nu)^p\,\pi_1(dx_1)\right)^{1/p}}{\left(\int  d(x_2,y)^p\,\pi_2(dx_2)\,\pi_2(dy)\right)^{1/p}}$$
and replacing $\mathcal{W}, \mathcal{AW}$ by the (adapted) $p$-Wasserstein distances $\mathcal{W}_p, \mathcal{AW}_p$ in all results. The restriction to $p=1$ is thus only chosen for notational simplicity.\\

This article is structured as follows: in Section \ref{sec:2} we derive basic properties of $\overrightarrow{\mathcal{W}}$ and compare it to other measures of association derived in \cite{chatterjee2020new, deb2020kernel} as well as Pearson's correlation coefficient in the case of a bivariate Gaussian distribution $\pi$. In Section \ref{sec:3} we state general continuity properties of the functional $\pi\mapsto \overrightarrow{\mathcal{W}}(\pi)$ with respect to $\mathcal{AW}$ and give a first consistency result. Section \ref{sec:4} and \ref{sec:5} then exhibit convergence rates for the independent case $\pi=\mu\otimes \nu$ and the general case respectively. Lastly Section \ref{sec:7} exhibits numerical simulations, while Section \ref{sec:8} discusses open research questions. We relegate longer proofs to the appendix.

\section{Discussion and literature review}\label{sec:2}

Recall that the Wasserstein correlation coefficient is given by
\begin{align*}
\overrightarrow{\mathcal{W}}(\pi)=\frac{\int \mathcal{W}(\pi_{x_1},\nu)\,\mu(dx_1)}{\int  d(y,z)\,\nu(dy)\,\nu(dz)}.
\end{align*}
We start with some preliminary comments: we first note that $\overrightarrow{\mathcal{W}}$ is not symmetric in the order of marginals, i.e. $ \overleftarrow{\mathcal{W}}(\pi)\neq \overrightarrow{\mathcal{W}}(\pi)$ in general, where
\begin{align*}
\overleftarrow{\mathcal{W}}(\pi):=\frac{\int \mathcal{W}(\pi_{y_1},\mu)\,\nu(dy_1)}{\int  d(x,z)\,\mu(dx)\,\mu(dz)}.
\end{align*}
This is showcased in the following example:

\begin{example}
Take $(\mathcal{X}, d)=(\R, |\cdot|_2)$ and set $$\pi=\frac{\delta_{(1,0)}+\delta_{(1,1)}+\delta_{(2,2)}}{3}.$$ Then 
\begin{align*}
\mu=\frac{2\delta_1+\delta_2}{3}, \qquad \nu=\frac{\delta_0+\delta_1+\delta_2}{3}
\end{align*} 
as well as 
\begin{align*}
\pi_{1}=\frac{\delta_0+\delta_1}{2}, \qquad \pi_{2} =\delta_2,
\end{align*}
so that $\mathcal{W}(\pi_{1},\nu)=1/2$, $\mathcal{W}(\pi_{2},\nu)=1$ and
\begin{align*}
\overrightarrow{\mathcal{W}}(\pi)=\frac{2/3}{8/9}=\frac{3}{4},
\end{align*}
while $\overleftarrow{\mathcal{W}}(\pi)$ equals $1$ by an obvious modification of Theorem \ref{thm:easy}.(iii) for $\overleftarrow{\mathcal{W}}$. In conclusion $\overleftarrow{\mathcal{W}}(\pi)\neq \overrightarrow{\mathcal{W}}(\pi)$.
\end{example}

The lack of symmetry is intentional: as we have already stated in Theorem \ref{thm:easy} we have $\overrightarrow{\mathcal{W}}=1$ iff $\nu=f_\#\mu$ for some measurable function $f$. The above example shows that this does not imply $\mu=g_\# \nu$ for some measurable function $g$. Nevertheless, in order to obtain a symmetric expression we could simply consider $\overleftarrow{\mathcal{W}}(\pi)\vee \overrightarrow{\mathcal{W}}(\pi)$ instead of merely $\overrightarrow{\mathcal{W}}(\pi)$. For notational simplicity we will only state our estimates for $\overrightarrow{\mathcal{W}}(\pi)$ and remark instead that all of them also hold for $\overleftarrow{\mathcal{W}}(\pi)$ as well as $\overleftarrow{\mathcal{W}}(\pi)\vee \overrightarrow{\mathcal{W}}(\pi)$, adjusting constants correspondingly. \\

Apart from the properties \textit{(i)-(iii)} stated in the introduction, \cite[Property (F)]{mori2019four} also asks for invariance properties of measures of association. In our case, the following can be directly derived from the definition of Wasserstein distances:

\begin{lemma}
Let $I:(\mathcal{X},d)\to (\mathcal{X},d)$ be a isometric isomorphism. Then $$\overrightarrow{\mathcal{W}}(\pi)=\overrightarrow{\mathcal{W}}((I,I)_\#\pi).$$ 
\end{lemma}

We now compare the Wasserstein correlation $\overrightarrow{\mathcal{W}}$ to several other measures of association. We start with the one proposed in \cite{trutschnig2011strong, dette2013copula, chatterjee2020new, junker2021estimating}:

\begin{lemma}\label{lem:chatterjee}
\cite{trutschnig2011strong, dette2013copula, chatterjee2020new, junker2021estimating}'s coefficient of correlation can be rewritten as
\begin{align*}
T^C(\pi)=\frac{ \int \int_{[0,1]} \left( F_{\mu_{x_1}}(F^{-1}_\nu(y))-y \right)^2\,dy\,\mu(dx_1)}{\int \mathrm{Var}\left( \mathds{1}_{\{Y\ge y\}}\right)\,\nu(dy)}
\end{align*}
where $F_{\mu_{x_1}}$ and $F^{-1}_\nu$ denotes the cdf of $\mu_{x_1}$ and $\nu$ respectively. In particular
\begin{align*}
\frac{\left(\int \mathcal{W}(\tilde{\pi}_{x_1}, \mathcal{U}([0,1]))\,\mu(dx_1)\right)^2}{\int \mathrm{Var}\left( \mathds{1}_{\{Y\ge y\}}\right)\,\nu(dy)} \le T^C(\pi)\le 2\frac{\int \mathcal{W}(\tilde{\pi}_{x_1}, \mathcal{U}([0,1]))\,\mu(dx_1)}{\int \mathrm{Var}\left( \mathds{1}_{\{Y\ge y\}}\right)\,\nu(dy)}
\end{align*}
for $\tilde{\pi}_{x_1}:=(F_\nu)_{\#}\pi_{x_1}$ and $\mathcal{U}([0,1])$ is the uniform distribution on $[0,1]$.
\end{lemma}

In particular $T^C$ can be estimated by the Wasserstein distance between $\tilde{\pi}_{x_1}$ and $\mathcal{U}([0,1])$. We note that compared to $\pi_{x_1}$ in the definition of $\overrightarrow{\mathcal{W}}$, $\tilde{\pi}_{x_1}$ is always compactly supported on $[0,1]$.\\

Next we compare $\overrightarrow{\mathcal{W}}$ to the functional obtained in \cite{deb2020kernel} for the specific case $(\mathcal{X},d)=(\R^d, |\cdot|_2)$:

\begin{lemma}\label{rem:1}
The coefficient of correlation obtained in  \cite{deb2020kernel} is given by
\begin{align*}
T^{\mathrm{DGS}}(\pi)&=1-\frac{\int |x_2-y|_2 \,\pi_{x_1}(dx_2)\,\pi_{x_1}(dy)\mu(dx_1)}{\int |y-z|_2 \,\nu(dy)\,\nu(dz)}\\
&= \frac{ \int |y-z|_2 \,\nu(dy)\,\nu(dz) - \int |x_2-y|_2 \,\pi_{x_1}(dx_2)\,\pi_{x_1}(dy)\,\mu(dx_1)}{\int |y-z|_2 \,\nu(dy)\,\nu(dz)}
\end{align*}
and $T^{\mathrm{DGS}}(\pi)\le 2\overrightarrow{\mathcal{W}}(\pi).$
\end{lemma}

By a similar reasoning, we can derive the following corollary:
\begin{corollary}
Let $(\mathcal{X},\|\cdot\|)$ be a normed space and let us define the measure of association derived from the norm $\|\cdot\|$ by 
\begin{align*}
T^{\|\cdot\|}(\pi)= 1-\frac{\int \|y-z\| \,\pi_{x_1}(dy)\,\pi_{x_1}(dz)\,\mu(dx_1)}{\int \|y-z\| \,\nu(dy)\,\nu(dz)}.
\end{align*}
Then we have $T^{\|\cdot\|}(\pi)\le 2 \overrightarrow{\mathcal{W}}(\pi).$ 
\end{corollary}

In particular all upper bounds derived in this article also hold for $T^{\|\cdot\|}(\pi)$, adjusting by a factor of $2$. However, the relation 
\begin{align*}
T^{\|\cdot\|}(\pi)=0 \text{ if and only if } \pi=\mu\otimes\nu
\end{align*}
might not hold, e.g. if $T^{\|\cdot\|}(\pi)$ only depends on a finite number of moments of $\pi$. Thus in general, the functional $\pi\mapsto \overrightarrow{\mathcal{W}}(\pi)$ offers greater flexibility than $\pi\mapsto T^{\|\cdot\|}(\pi)$ as it can be defined for any metric $d$ instead of just any norm $\|\cdot\|$, while it always satisfies the properties \textit{(i)}-\textit{(iii)} of Theorem \ref{thm:easy}.\\

Next we compare $\overrightarrow{\mathcal{W}}(\pi)$ to a (non-normalised version) of the Hellinger correlation introduced in \cite{geenens2020hellinger} for $\mu,\nu\in \mathrm{Prob}(\R)$. As this correlation is based on the Hellinger distance, singular measures are slightly intricate to handle. To avoid technicalities we thus only consider the following simple case:

\begin{lemma}\label{lem:hell}
Assume that the probability measures $\mu,\nu\in \text{Prob}(\R)$ and $\pi\in \Pi(\mu,\nu)$ have densities $f_\mu, f_\nu, f_\pi$ wrt. the Lebesgue measure. Then the Hellinger correlation $T^H(\pi)$ of \cite[Section 4]{geenens2020hellinger} can be written as
\begin{align*}
T^H(\pi)=\int \int \left( \sqrt{ f_\pi(x_1, x_2)}-\sqrt{ f_\mu(x_1) f_\nu(x_2)} \right)^2\,dx_1\,dx_2.
\end{align*}
If $\mu,\nu$ have bounded support, then there exists a constant $C>0$ such that
\begin{align*}
\overrightarrow{\mathcal{W}}(\pi) \le C \frac{T^H(\pi)}{\int |y-z|_2\,\nu(dy)\,\nu(dz)} .
\end{align*}
\end{lemma}

Lastly let us compare the Wasserstein correlation to a classical benchmark: recall that if $\pi$ is a bivariate Gaussian distribution, then the association between $\mu$ and $\nu$ is famously quantified via Pearson's correlation coefficient. It turns out that we can also compute $\overrightarrow{\mathcal{W}}(\pi)$ explicitly in this case:

\begin{lemma}[Comparison with Pearson's correlation coefficient in the case $p=2$] \label{lem:pearson}
Let $(\mathcal{X},d)=(\R,|\cdot|)$ and let $\pi=\mathcal{N}(a, \Sigma)$, where
$a=(a_1, a_2)$ is the mean and $$\Sigma=\begin{bmatrix}
&\sigma_1^2& \rho \sigma_1\sigma_2\\
&\rho \sigma_1 \sigma_2 &\sigma_2^2
\end{bmatrix}$$
is the variance of the bivariate normal distribution $\pi$. Here we assume $\sigma_1, \sigma_2> 0$ and note that $\rho\in [-1,1]$ is Pearson's correlation coefficient. Then $\overrightarrow{\mathcal{W}}_2(\pi)=1-\sqrt{1-\rho^2}.$
\end{lemma}

We now compare different coefficients of correlation for the bivariate standard normal case in Figure \ref{fig:1}. As anticipated in Section \ref{sec:1} the Wasserstein correlation is convex in $\rho$.

\begin{center}
\begin{figure}[h!]
\includegraphics[scale=0.4]{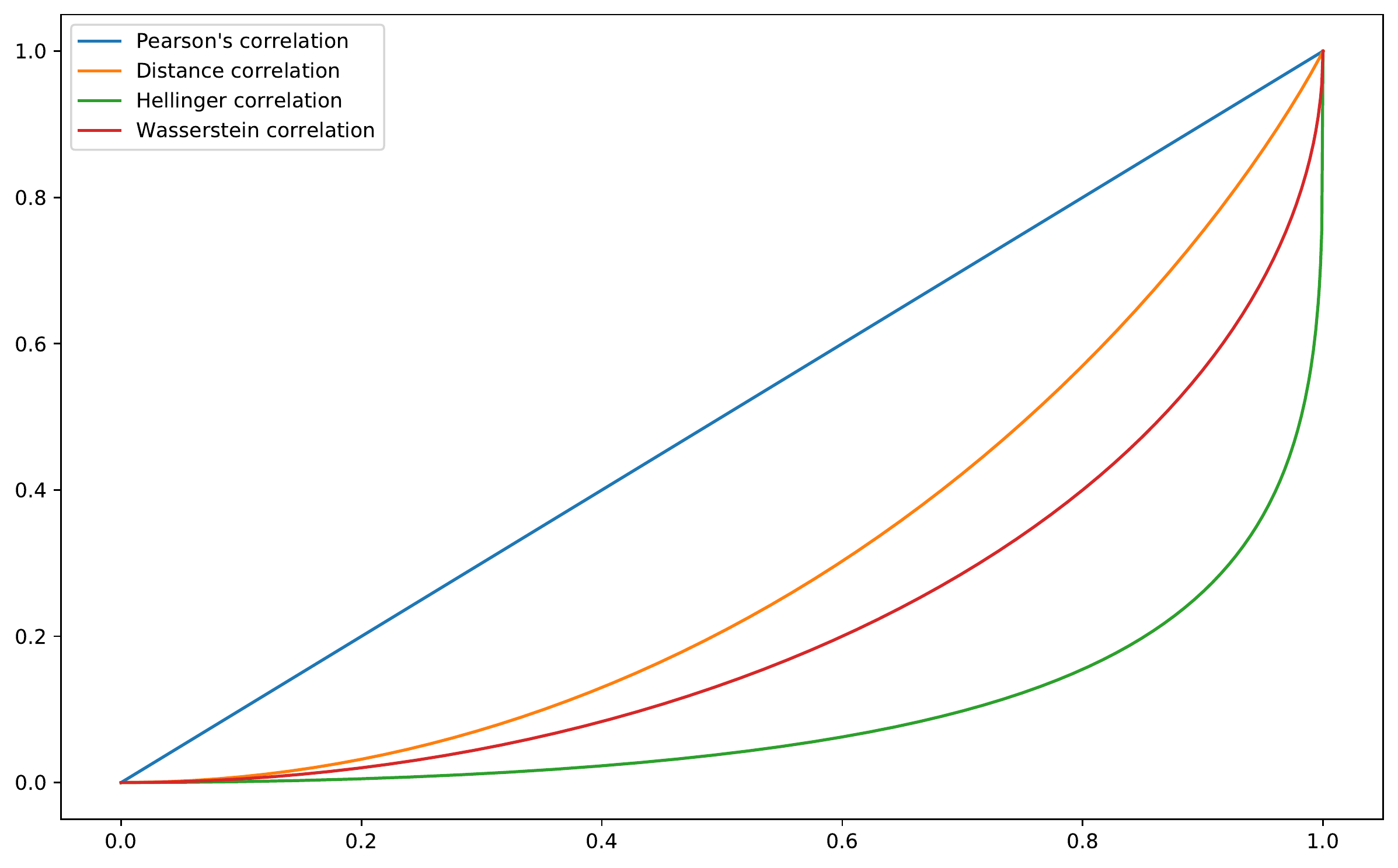}
\caption{Comparison of correlation coefficients as a function of $\rho$ for $\rho$-correlated bivariate Gaussian distribution}
\end{figure}\label{fig:1}
\end{center}

\begin{remark}
It turns out that the term $\sqrt{1-\rho^2}$ is the geometric mean of the eigenvalues of $\Sigma$, which can be interpreted as a geometric proxy for dependence. Unfortunately such an interpretation is no longer obvious in higher dimensions. Indeed for a bivariate normal distribution with covariance matrix
$$\Sigma=\begin{bmatrix}
&\Sigma_{11}& \Sigma_{12}\\
&\Sigma_{21} & \Sigma_{22}
\end{bmatrix}$$
we still have
\begin{align*}
\pi_{x_1}=\mathcal{N}\left( a_2+\Sigma_{21}\Sigma_{11}^{-1} (x_1-a_1), \Sigma_{22}-\Sigma_{21}\Sigma_{11}^{-1}\Sigma_{12} \right),
\end{align*}
but in general dimensions we obtain
\begin{align*}
\mathcal{W}_2(\pi_{x_1},\nu)^2 &= \left|\Sigma_{21} \Sigma_{11}^{-1} (x_1-a_1) \right|^2 +\mathrm{tr}\left(\Sigma_{22}\right)+\mathrm{tr}\left( \Sigma_{22}-\Sigma_{21}\Sigma_{11}^{-1}\Sigma_{12}\right)\\
&\quad -2 \mathrm{tr}\left(\left( \Sigma_{22}^2-\Sigma_{21}\Sigma_{11}^{-1}\Sigma_{12}\Sigma_{22} \right)^{1/2}\right).
\end{align*}
\end{remark}

\section{An estimator for the Wasserstein correlation and its asymptotic consistency}\label{sec:3}

We now investigate continuity properties of the functional $\pi\mapsto \overrightarrow{\mathcal{W}}(\pi)$, which will enable us to construct a plugin estimator. We then check its asymptotic consistency.\\\\
Let us first state that the functional $\pi\mapsto \overrightarrow{\mathcal{W}}(\pi)$ is continuous in the adapted Wasserstein distance $\mathcal{AW}$:

\begin{theorem}\label{thm:triangle}
For $\pi\in \Pi(\mu,\nu)$ and $\tilde{\pi}\in \Pi(\tilde{\mu},\tilde{\nu})$ we have 
\begin{align*}
\left| \int \mathcal{W}(\pi_{x_1}, \nu)\,\mu(dx_1) - \int \mathcal{W}(\tilde{\pi}_{y_1}, \tilde{\nu})\,\tilde{\mu}(dy_1) \right|&\le  \mathcal{AW}(\pi, \tilde{\pi})+\mathcal{W}(\nu,\tilde{\nu})\\
&\le 2 \mathcal{AW}(\pi, \tilde{\pi}) 
\end{align*}
and thus in particular
\begin{align*}
\left| \overrightarrow{\mathcal{W}}(\pi)-\overrightarrow{\mathcal{W}}(\tilde{\pi}) \right|&\le \frac{1}{f(\tilde{\nu})}\Big(\mathcal{AW}(\pi,\tilde{\pi})+\mathcal{W}(\nu,\tilde{\nu})+ g(\nu,\tilde{\nu}) \Big)\\
&\le \frac{1 }{f(\tilde{\nu})}\Big(\mathcal{AW}(\pi,\tilde{\pi})+3\mathcal{W}(\nu,\tilde{\nu}) \Big)\\
&\le \frac{4 }{f(\tilde{\nu})}\, \mathcal{AW}(\pi,\tilde{\pi}),
\end{align*}
where
\begin{align*}
f(\tilde{\nu})&:= \int  d(y,z)\,\tilde{\nu}(dy)\,\tilde{\nu}(dz) \\
g(\nu, \tilde{\nu})&:=\left|\int  d(y,z)\,\tilde{\nu}(dy)\,\tilde{\nu}(dz) - \int  d(y,z)\,\nu(dy)\,\nu(dz)  \right|.
\end{align*}
\end{theorem}

We have the following immediate corollary:

\begin{corollary}\label{cor:consistency}
Let $\pi \in \Pi(\mu,\nu)$ such that 
\begin{align*}
\int d(x_2,x_0)\,\nu(dx_2) <\infty
\end{align*}
for any $x_0\in \mathcal{X}$ and let $\hat{\pi}^N$ be an $\mathcal{AW}$-consistent estimator of $\pi$. Then $\overrightarrow{\mathcal{W}}(\hat{\pi}^N)$ is an asymptotically consistent estimator of $\overrightarrow{\mathcal{W}}(\pi)$.
\end{corollary}

\begin{proof}
Theorem \ref{thm:triangle} yields
\begin{align*}
\begin{split}
\left| \overrightarrow{\mathcal{W}}(\pi)-\overrightarrow{\mathcal{W}}(\hat{\pi}^N) \right|
&\le \frac{4 }{f(\hat{\pi}^N_2)} \, \mathcal{AW}(\pi,\hat{\pi}^N).
\end{split}
\end{align*}
By assumption we have $\lim_{N\to \infty} \mathcal{AW}(\pi,\hat{\pi}^N)=0$. By the proof of Theorem \ref{thm:triangle} in the appendix we conclude that $g(\nu, \hat{\pi}^N_2)\le 2 \mathcal{AW}(\pi,\hat{\pi}^N)$ so that $$\lim_{N\to \infty} f(\hat{\pi}^N_2)=f(\nu),$$ where $f(\nu)>0$ by assumption. This concludes the proof.
\end{proof}

We now give an explicit example of an $\mathcal{AW}$-consistent estimator $\hat{\pi}^N$, which will then naturally facilitate a plugin estimator $\overrightarrow{\mathcal{W}}(\hat{\pi}^N)$ for $\overrightarrow{\mathcal{W}}(\pi)$. We only discuss here the case where $\pi$ is a probability measure on $([0,1]^d)^2$, where we equip $[0,1]^d$ with the Euclidean metric $|\cdot|_2$. Of course, our analysis can then easily be extended to probability measures on any compact subset of $\R^d$. Determining an explicit $\mathcal{AW}$-consistent estimator for non-compactly supported $\pi$ is still an open question and left for future research.\\ 
Before we explain the details of the construction, we need to introduce some additional notation: for a subset $F$ of $\R^d$ let  $\mathop{\mathrm{diam}} (F):=\sup_{x,y\in F} |x-y|_2$ and for any set A, let $|A|$ denote the number of elements in $A$. Lastly, for any $\pi\in \text{Prob}(([0,1]^d)^2)$ and any Borel set $G\subseteq [0,1]^d$ we define the conditional probability
\begin{align*}
\pi_G(\cdot)=\frac{1}{\pi_1(G)} \int_G \pi_{x_1}(\cdot)\,\pi_1(dx_1)\in \text{Prob}([0,1]^d),
\end{align*}
where we make the convention that $\pi_G:=\delta_{0}$ if $\pi_1(G)=0$.
\\
 
Let us assume that we are given i.i.d. samples $(X_1^1, X_2^1), (X_1^2, X_2^2), \dots, (X_1^N, X_2^N) $ of $\pi$. Let us partition the unit cube $[0,1]^d$ into a disjoint union of a finite number of cubes and let $\varphi^N \colon[0,1]^d\to[0,1]^d$ map each cube to its center. Then in particular $\varphi^N$ has a finite range for each $N\ge 1$. We now set
	\[ \picausal:=\frac{1}{N}\sum_{n=1}^N \delta_{\varphi^N(X_1^n),\varphi^N(X_2^n)} \]
	for each $N\geq 1$. In other words, if we define
\[\Phi^N:=\big\{(\varphi^N)^{-1}(\{x\}) : x\in \varphi^N([0,1]^d) \big\},\]
then
\[ [0,1]^d=\bigcup_{G\in \Phi^N} G \quad\text{disjoint.}
\]

One of the main results of \cite{backhoff2020estimating} is the following:
\begin{lemma}[{\cite[Theorem 1.3]{backhoff2020estimating}}]
\label{thm:almost.sure.convergence}
	Assume that $\lim_{N\to \infty} |\Phi^N|/N=0$. Then the adapted empirical measures is a strongly consistent estimator, that is, 
	\[  \lim_{N\to\infty} \mathcal{AW}(\pi,\picausal) = 0 \]	
	$P$-almost surely.
\end{lemma}

\section{The case of independent marginals: $\pi=\mu\otimes\nu$}\label{sec:4}

In this section we discuss convergence rates of $\overrightarrow{\mathcal{W}}(\picausal)$ for the case $\pi=\mu\otimes\nu$. We then show how to construct a test for independence of $\mu$ and $\nu$ using the estimator $\overrightarrow{\mathcal{W}}(\picausal)$. As $\overrightarrow{\mathcal{W}}(\pi)\in [0,1]$ for all $\pi\in \Pi(\mu,\nu)$ we cannot hope for a CLT as in \cite[Theorem 4.1]{deb2020kernel}. However, we can still obtain parametric convergence rates. Indeed, the main insight of this section is the following result:

\begin{theorem}\label{thm:test1}
If $\pi=\mu\otimes \nu$ then we have for all $\epsilon>0$
\begin{align*}
\P\left( \int \mathcal{W}\left(\picausal[x_1],\picausal[]_2\right)\,\picausal[]_1(dx_1) \ge \epsilon\right)&\le \exp\left( \log(2) \frac{(|\Phi^N|+1)^2}{N}-\frac{\varepsilon^{2}N}{ 2d}\right)
\end{align*}
and consequently 
\begin{align*}
\overrightarrow{\mathcal{W}}(\picausal) =O_P\left(\frac{|\Phi^N|}{\sqrt{N}}\right).
\end{align*}
In particular, if $$\lim_{N\to \infty} \frac{|\Phi^N|}{\sqrt{N}}=0\quad \text{and}\quad \lim_{N\to \infty} \frac{|\Phi^N|^2}{\log N}=\infty,$$ then there exists $C=C(\nu)>0$ such that the test: reject $\pi=\mu\otimes \nu$ if
\begin{align*}
\overrightarrow{\mathcal{W}}(\picausal)> C \frac{|\Phi^N|}{\sqrt{N}},
\end{align*}
satisfies the following:
\begin{itemize}
\item if $\pi=\mu\otimes \nu$ then there exists $N_0=N_0(\omega)\in \N$ such that $\overrightarrow{\mathcal{W}}(\picausal)\le C \frac{|\Phi^N|}{\sqrt{N}}$ for all $N\ge N_0$,
\item if $\pi\neq \mu\otimes \nu$, then there exists $N_0=N_0(\omega)\in \N$ such that $\overrightarrow{\mathcal{W}}(\picausal)> C \frac{|\Phi^N|}{\sqrt{N}}$ for all $N\ge N_0$.
\end{itemize}
\end{theorem}

We note here that as the construction of $\picausal$ is fully explicit and no additional assumptions on the measure $\pi$ are necessary, which makes the above result conceptually easy to apply.\\

Lastly, we can construct the following simple test statistic for independence of $\mu$ and $\nu$:

\begin{corollary}\label{cor:go}
Under the assumptions that $\mu$ and $\nu$ are non-atomic and $\pi=\mu\otimes\nu$, there exists a constant $C(\nu)$ such that the test: reject $\pi=\mu\otimes \nu$ if 
\begin{align*}
\overrightarrow{\mathcal{W}}(\picausal)> C(\nu)\left( \sqrt{\frac{2}{\pi}} \frac{|\Phi^N|}{\sqrt{N}}+\frac{\sigma}{\sqrt{N}} \Phi^{-1}(1-\alpha)\right),
\end{align*}
where $\Phi^{-1}$ denotes the quantile function of the standard normal distribution and $\sigma=1-2/\pi$, has asymptotic significance level $\alpha$.
\end{corollary}

\begin{proof}
As in the proof of Theorem \ref{thm:test1}, this follows from the inequality
\begin{align*}
\overrightarrow{\mathcal{W}}(\picausal)\le \frac{\sqrt{d}\, \tilde{T}_N(\pi)}{ \int |x_2-y|\,\picausal[]_2(dx_2)\,\picausal[]_2(dy)}\le C(\nu) \tilde{T}_N(\pi)
\end{align*}
for some $C(\nu)>0$, which holds for all sufficiently large $N\in \N$. Here $\tilde{T}_N$ is given by 
\begin{align*}
\tilde{T}_N(\pi)&:=
\sum_{G\in \Phi^N} \sum_{H\in\Phi^N} \Bigg|\frac{|n\in \{1,\dots, N\} \text{ s.t. } X_1^n\in G, X_2^n\in H|}{N}\\
&\qquad\qquad- \frac{|n\in \{1, \dots, N\} \text{ s.t. } X_1^n\in G|}{N}\cdot \frac{|n\in \{1,\dots, N\} \text{ s.t. } X_2^n\in H|}{N}\Bigg|.
\end{align*}
We the conclude by Lemma \cite[Lemma 1.1]{jwb}. 
\end{proof}

\section{General convergence rates for the Wasserstein correlation}\label{sec:5}

We now derive general rates of convergence for $\overrightarrow{\mathcal{W}}(\pi)$, using results recently obtained in \cite{backhoff2020estimating}. In particular we slightly refine the definition of $\varphi^N$ and thus the adapted empirical measure $\picausal$ given in Section \ref{sec:3} as follows: we set $r=1/3$ for $d=1$ and $r=1/(2d)$ for all $d\geq 2$.
	For all $N\geq 1$, let us  now partition the cube $[0,1]^d$ into the disjoint union of $N^{rd}$ cubes with edges of length $N^{-r}$ and let $\varphi^N\colon[0,1]^d\to[0,1]^d$ map each such small cube to its center. As before we then set
	\[ \picausal:=\frac{1}{N}\sum_{n=1}^N \delta_{\varphi^N(X_1^n),\varphi^N(X_2^n)}. \]
	for each $N\geq 1$.

In order to quantify the speed of convergence, we assume the following regularity property on the disintegration of $\pi$ for the remainder of this section:

\begin{assumption}[Lipschitz kernels]
\label{ass:lipschitz.kernel}
	There is a version of the ($\mu$-a.s.\ uniquely defined) disintegration such that the mapping
	\begin{align*}
	([0,1]^d)
	\ni x_1
	\mapsto \pi_{x_1}
	\in \mathrm{Prob}([0,1]^d)
	\end{align*}
	is Lipschitz continuous, where $\mathrm{Prob}([0,1]^d)$ is endowed with its usual Wasserstein distance $\mathcal{W}$.
\end{assumption} 

\cite[Lemma 1.3]{jwb} and \cite[Lemma 1.4]{jwb} stated in the appendix together with Theorem \ref{thm:triangle} immediately enable us to deduce average convergence rates and a deviation result for the plugin estimator $\overrightarrow{\mathcal{W}}(\picausal)$. More concretely we obtain the following:

\begin{theorem}
Under Assumption \ref{ass:lipschitz.kernel}, there is a constant $C>0$ such that 
\begin{align*}
&\E \left| \int \mathcal{W}\left(\picausal[x_1],\picausal[]_2\right)\,\picausal[]_1(dx_1) -\int \mathcal{W}\left(\pi_{x_1},\nu\right)\,\mu(dx_1) \right|\\
&\qquad\le C(\nu) \cdot
	\begin{cases}
	N^{-1/3} &\text{for } d=1,\\
	N^{-1/4}\log(N+1) &\text{for } d=2,\\
	N^{-1/(2d)} &\text{for } d\geq 3.
	\end{cases} \\
\end{align*}
In particular we have
\begin{align*}
\left| \overrightarrow{\mathcal{W}}(\picausal) -\overrightarrow{\mathcal{W}}(\pi)\right| =
	\begin{cases}
	O_P(N^{-1/3}) &\text{for } d=1,\\
	O_P(N^{-1/(2d)}) &\text{for } d\geq 2.
	\end{cases} 
\end{align*}
\end{theorem}

\begin{remark}
We note that both \cite[Theorem 5.1]{dette2013copula} as well as \cite[Theorem 2.2]{chatterjee2020new} (under the assumption that $\pi=\mu\otimes\nu$) derive a CLT for their corresponding estimators of the correlation coefficient $T^{DGSC}$, while our convergence rates suffer from the ``curse of dimensionality" -- a by now well-studied phenomenon for Wasserstein distances (see e.g. \cite{fournier2015rate,weed2017sharp}). In particular \cite{dette2013copula,chatterjee2020new} achieve strictly better convergence rates than we do even if $d=1$, at least if $\pi\neq \mu\otimes \nu$. We recall however that $T^{DGSC}$ is only well-defined for a 1-dimensional distribution $\nu\in \mathrm{Prob}(\R)$ and cannot easily be generalised. In contrast, the main goals of this article was to provide a correlation coefficient which can be defined on arbitrary Polish spaces $\mathcal{X}$.
\end{remark}


\begin{proof}
By Theorem \ref{thm:triangle} we have
\begin{align*}
\int \mathcal{W}\left(\picausal[x_1],\picausal[]_2\right)\,\picausal[]_1(dx_1) -\int \mathcal{W}\left(\pi_{x_1},\nu\right)\,\mu(dx_1)\le 2 \mathcal{AW}(\pi,\picausal),
\end{align*}
so the first claim follows from \cite[Lemma 1.3]{jwb} replacing $C$ by $2C$.
Moreover, Theorem \ref{thm:triangle} also states that
\begin{align*}
\left| \overrightarrow{\mathcal{W}}(\pi)-\overrightarrow{\mathcal{W}}(\picausal) \right|\le \frac{4 }{f((\picausal)^2)}\, \mathcal{AW}(\pi,\picausal)
\end{align*}
Combining this with \cite[Lemma 1.2]{jwb}, which states that
\begin{align*}
\left(\sqrt{N}\wedge \frac{1}{2\cdot \sup_{x} |\varphi^N(x)-x|} \right)\left(f(\picausal[]_2)-f(\nu)\right)
&=\left(\sqrt{N}\wedge \frac{N^r}{2}\right)\left(f(\picausal[]_2)-f(\nu)\right)\\
&= \frac{N^r}{2}\left(f(\picausal[]_2)-f(\nu)\right)=O_P(1),
\end{align*}
and as $f(\nu)>0$, 
we again obtain the claim from \cite[Lemma 1.3]{jwb}.
This concludes the proof.
\end{proof}

In a similar fashion we can derive concentration bounds from \cite[Lemma 1.4]{jwb}:

\begin{theorem}
Under Assumption \ref{ass:lipschitz.kernel} there are constants $c,C>0$ such that
\begin{align*}
&P\left[ \left| \int \mathcal{W}\left(\picausal[x_1],\picausal[]_2\right)\,\picausal[]_1(dx_1) -\int \mathcal{W}\left(\pi_{x_1},\nu\right)\,\mu(dx_1) \right|  \geq 2C\mathop{\mathrm{rate}}(N)+\varepsilon \right]\\
	&\quad \leq 4\exp\Big( -cN\varepsilon^2 \Big) 
\end{align*} 
	for all $N\geq 1$ and all $\varepsilon>0$.
\end{theorem}

\begin{proof}
Using again Theorem \ref{thm:triangle}, this time together with Lemma \cite[Lemma 1.4]{jwb}, we obtain the existence of two constants $c,C>0$
\begin{align*}
&P\left[ \left| \int \mathcal{W}\left(\picausal[x_1],\picausal[]_2\right)\,\picausal[]_1(dx_1) -\int \mathcal{W}\left(\pi_{x_1},\nu\right)\,\mu(dx_1) \right|  \geq 2C\mathop{\mathrm{rate}}(N)+\varepsilon \right]\\
 & \le P\Big[ 2\mathcal{AW}(\mu,\picausal)  \geq 2C\mathop{\mathrm{rate}}(N)+\varepsilon \Big]\\
	& \leq 4\exp\Big( -cN\varepsilon^2 \Big) 
\end{align*}
replacing $c$ by $c/4$ in \cite[Lemma 1.4]{jwb}. This concludes the proof.
\end{proof}

\section{Numerical examples}\label{sec:7}

In this section we numerically compare the Wasserstein correlation with other measures of dependence. In particular we exhibit the behaviour of the Pearson, Spearman, Chatterjee, distance and Wasserstein correlation. The code for the numerical implementation can be found under \href{https://github.com/johanneswiesel/Wasserstein-correlation}{https://github.com/johanneswiesel/Wasserstein-correlation}. In particular we use the Xicor package (\href{https://github.com/czbiohub/xicor/}{https://github.com/czbiohub/xicor/}) to compute Chattterjee's correlation coefficient. We use the GitHub depository satra/distcorr.py (\href{https://gist.github.com/satra/aa3d19a12b74e9ab7941}{https://gist.github.com/satra/aa3d19a12b74e9ab7941}) to compute the distance correlation and the POT package (\href{https://pythonot.github.io}{https://pythonot.github.io}) to compute the Wasserstein correlation efficiently using entropic penalisation. For simplicity we keep the penalisation parameter fixed at level $\epsilon=0.01$. While entropic penalisation considerably speeds up the computation compared to the exact solution of the optimal transport problem, the computational burden of the Wasserstein correlation is nevertheless non-negligible compared to the other measures of association discussed here. Mitigating these computational disadvantages (particularly in high dimensions) is however an active topic of research. In the case of $\overrightarrow{\mathcal{W}}$ it might also be overcome by using an estimation procedure specifically adapted for the Wasserstein correlation. \\
 
In order to compare the behaviour of the correlation coefficients listed above for different kinds of dependence between $X_1$ and $X_2$, we plot four examples in Figures \ref{fig:2} and \ref{fig:3}: in all cases we take $X_1$ to be uniformly distributed on the interval $[0,1]$, while $X_2= f(\rho X_1+\sqrt{1-\rho^2} U)$ for an independent random variable $U$ with the same law as $X_1$. We consider the functions $f(x)=x$, $f(x)=|x-0.5|$, $f(x)=(x-0.5)^3$ and $f(x)=\sin(3x)$. We take $N=1000$ samples and plot the average over $30$ different draws below.

\begin{figure}[h!]
\begin{center}
\begin{minipage}{0.48\textwidth}
\includegraphics[scale=0.25]{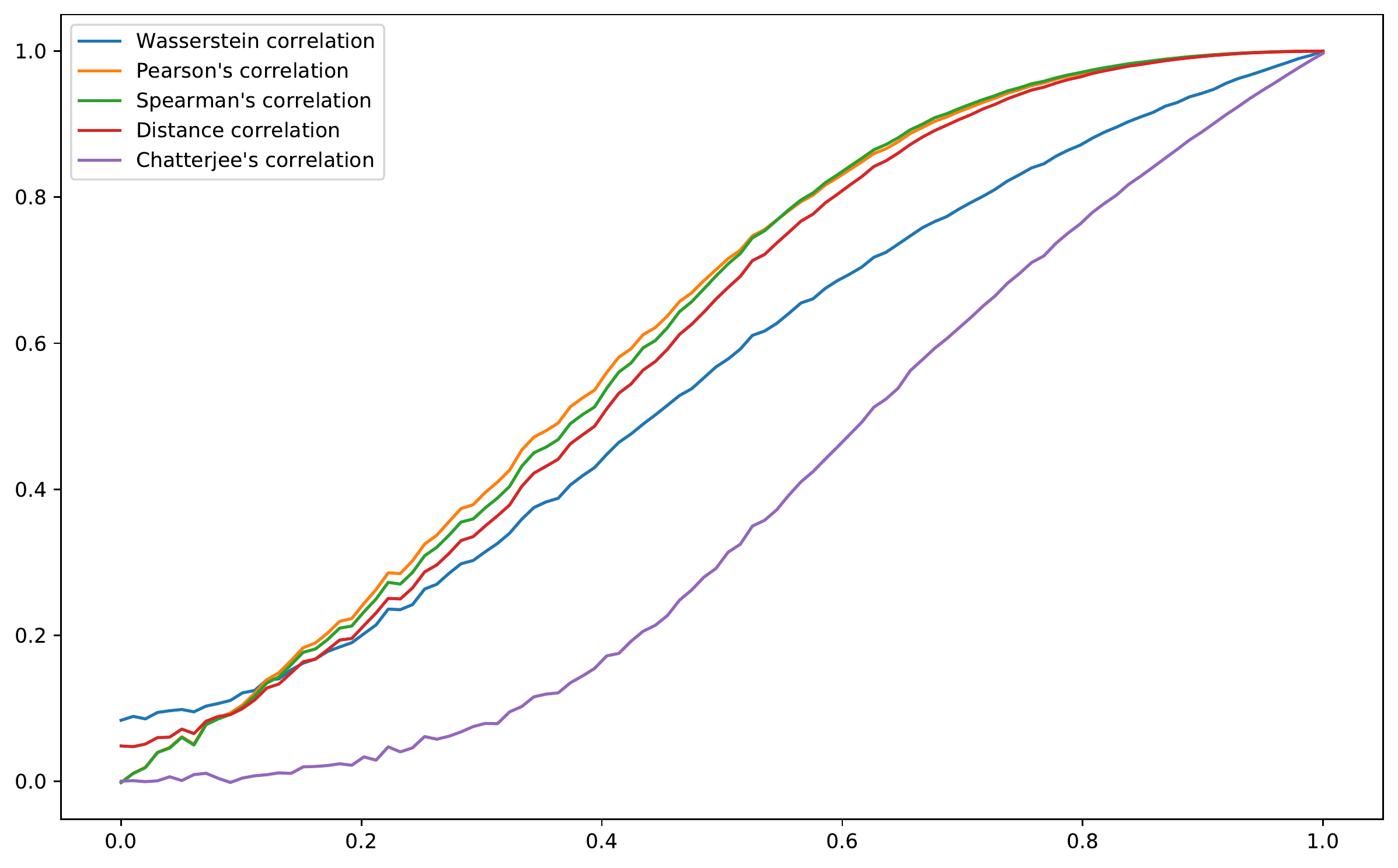}
\end{minipage}
\quad
\begin{minipage}{0.48\textwidth}
\includegraphics[scale=0.25]{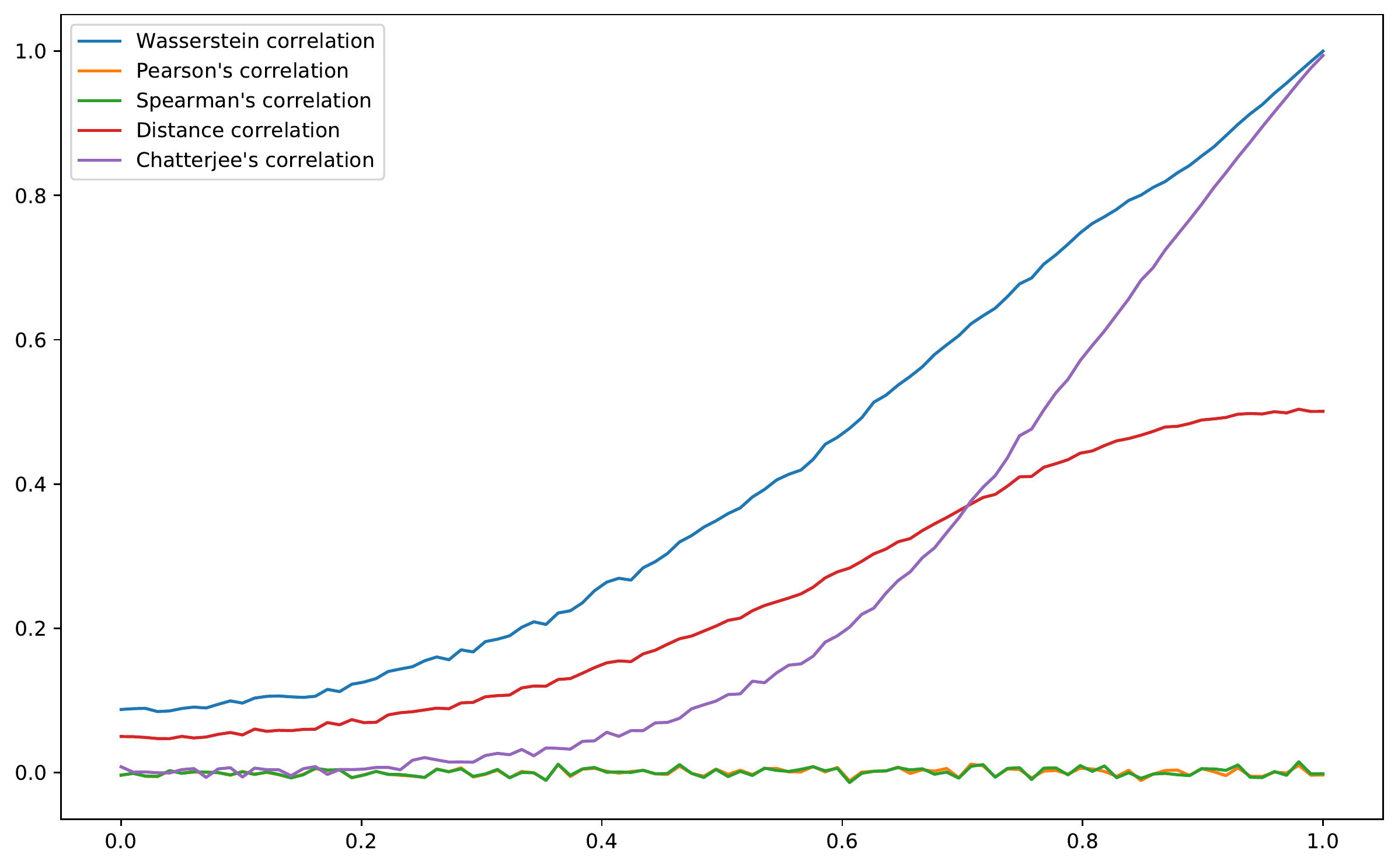}
\end{minipage}
\caption{Comparison of correlation coefficients as a function of $\rho$ for $\rho$-correlated bivariate uniform distribution with $f(x)=x$ (left) and $f(x)=|x-0.5|$ (right). Both plots use $N=1000$ samples and show the average over $30$ different draws.}\label{fig:2}
\end{center}
\end{figure}

\begin{figure}[h!]
\begin{center}
\begin{minipage}{0.48\textwidth}
\includegraphics[scale=0.25]{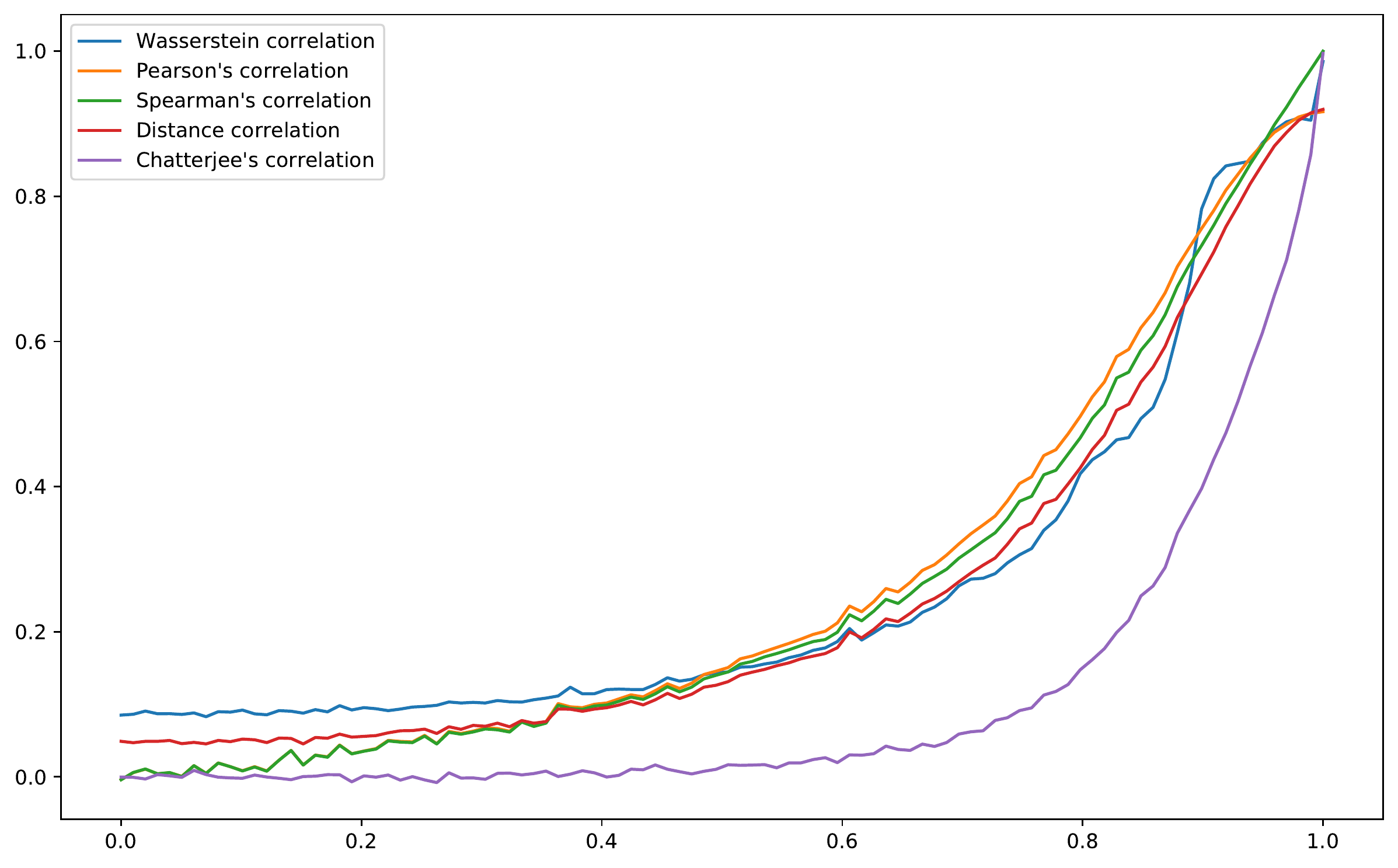}
\end{minipage}
\quad
\begin{minipage}{0.48\textwidth}
\includegraphics[scale=0.25]{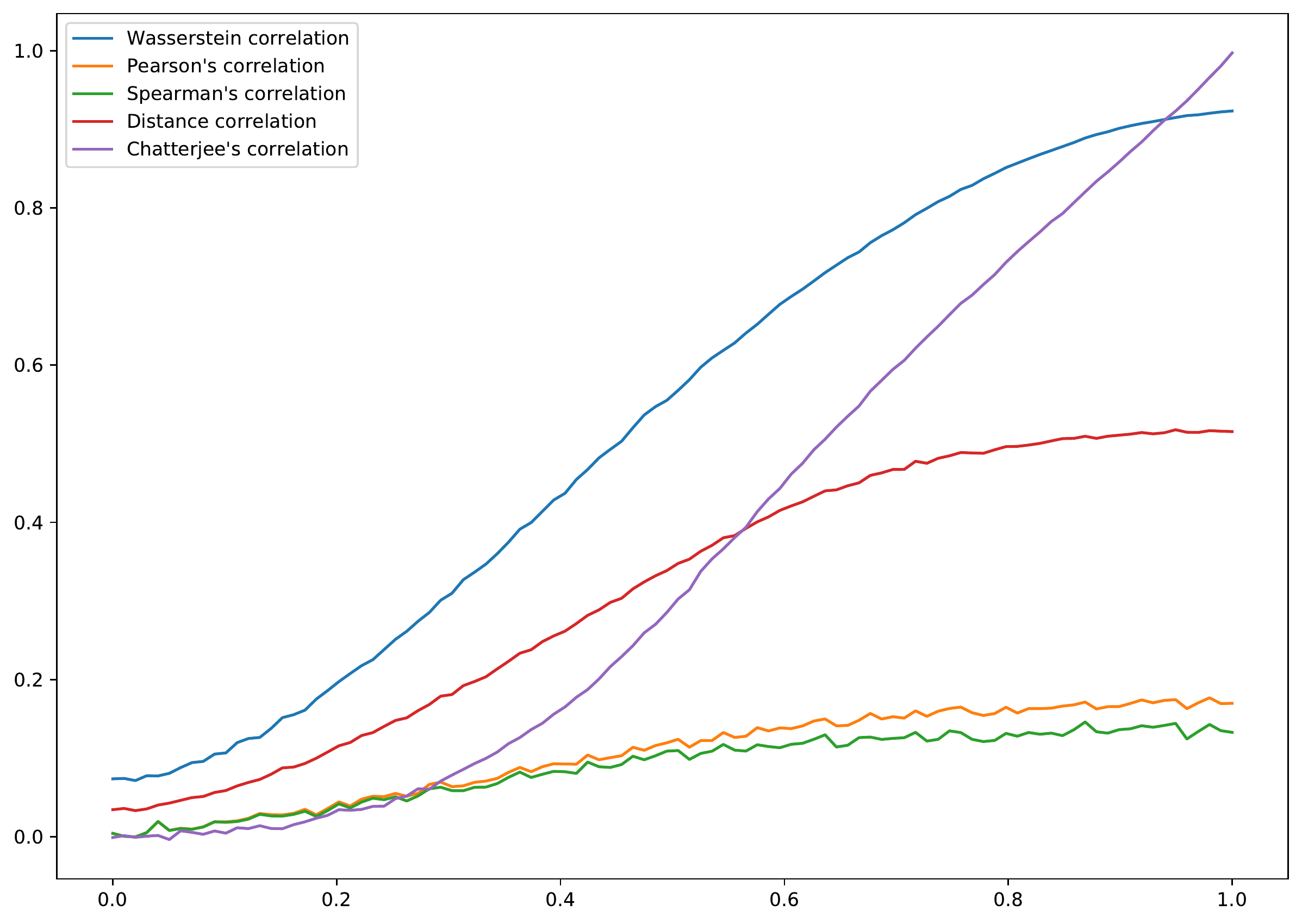}
\end{minipage}
\caption{Comparison of correlation coefficients as a function of $\rho$ for $\rho$-correlated bivariate uniform distribution with $f(x)=(x-0.5)^3$ (left) and $f(x)=\sin(3x)$ (right). Both plots use $N=1000$ samples and show the average over $30$ different draws.}\label{fig:3}
\end{center}
\end{figure}

While the traditional coefficients perform slightly better for near-linear relationships between $X_1$ and $X_2$, the strength of the Wasserstein correlation becomes evident in the non-linear setting: for the functions $f(x)=|x-0.5|$ and $f(x)=\sin(3x)$, where Pearson and Spearman's correlation do not pick up any functional relationship, the Wasserstein correlation increases faster than distance correlation and Chatterjee's correlation.\\

Next we discuss an application to historical data: we compare our estimate for the Wasserstein correlation with other correlation coefficients on Sir Francis Galton's peas data. The estimator $T^C$ was already computed in \cite[Section 3]{chatterjee2020new} and we refer to the exposition there for a historical discussion. The data consists of 700 pairs $(X_1, X_2)$ of mean diameters of sweet peas of mother plants and daughter plants. We list the results in Figure \ref{fig:4}.

\begin{figure}[h!]
\centering
\begin{tabular}{cc}
Correlation coefficient & Value  \\ 
  \hline
 Pearson's correlation & 0.3463  \\ 
 Spearman's correlation & 0.3615\\
 Distance correlation & 0.3216\\
 Chatterjee's correlation & 0.1186 \\
 Wasserstein correlation & 0.3124  
\end{tabular}
\qquad\qquad
\begin{tabular}{cc}
Correlation coefficient & Value  \\ 
  \hline
 Pearson's correlation & 0.3463  \\ 
 Spearman's correlation & 0.3615 \\
 Distance correlation & 0.3216\\
 Chatterjee's correlation & 0.9224 \\
 Wasserstein correlation & 0.9409  
\end{tabular}
\caption{Correlation between $X_1$ and $X_2$ (left) and $X_2$ and $X_1$ (right) for different correlation coefficients}\label{fig:4}
\end{figure}

We observe that the Wasserstein correlation between $X_1$ and $X_2$ is only slightly lower than classical coefficients like Pearson's or Spearman's correlation, while Chatterjee's correlation with a value of 0.12 (averaged over $10^5$ samples) is quite low. When reversing the roles of $X_1$ and $X_2$, both the Wasserstein correlation and Chatterjee's correlation increase decisively, hinting at the deterministic relationship $X_1=f(X_2)$ for some function $f$. As in \cite{chatterjee2020new}, this can be made plausible by the specific structure of the data.\\

Now we consider the case of independent marginals $\pi=\mu\otimes \nu$ of Section \ref{sec:4} and compare the power functions of four different tests for independence: these are derived from our Corollary \ref{cor:go}, \cite[Theorem 2.1]{chatterjee2020new}, \cite[Theorem 6]{szekely2007measuring} and \cite[Theorem 3.1]{shi2020distribution} respectively. We estimate the power function on 1500 samples over 200 draws. We remind the reader however that all four tests only have a theoretical foundation for the asymptotic regime $N\to \infty$. Despite our best efforts we have not been able to resolve the numerical instabilities occuring in the power of Chatterjee's test using the Xicor package.

\begin{figure}[h!]
\begin{center}
\begin{minipage}{0.48\textwidth}
\includegraphics[scale=0.27]{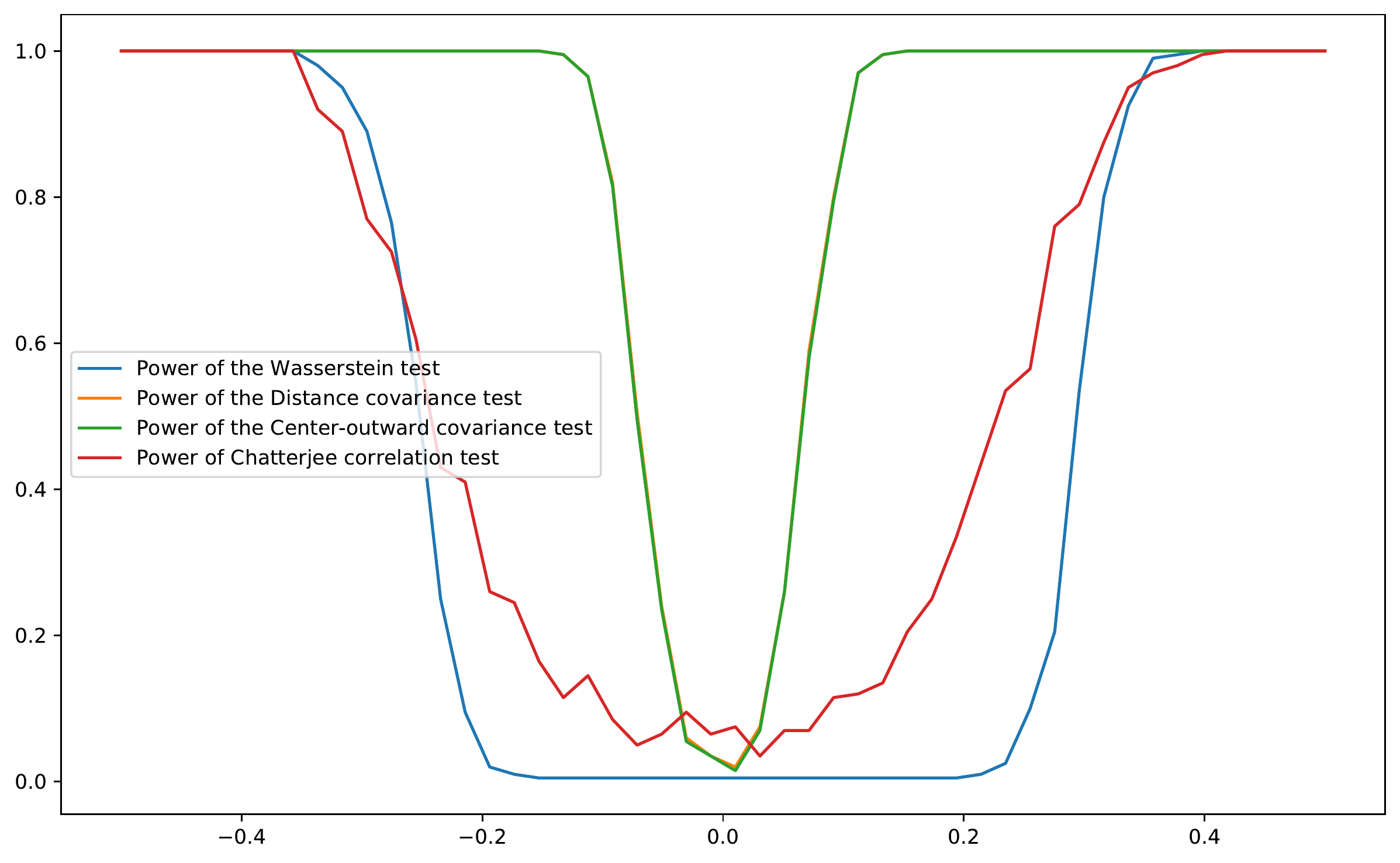}
\end{minipage}
\quad
\begin{minipage}{0.48\textwidth}
\includegraphics[scale=0.27]{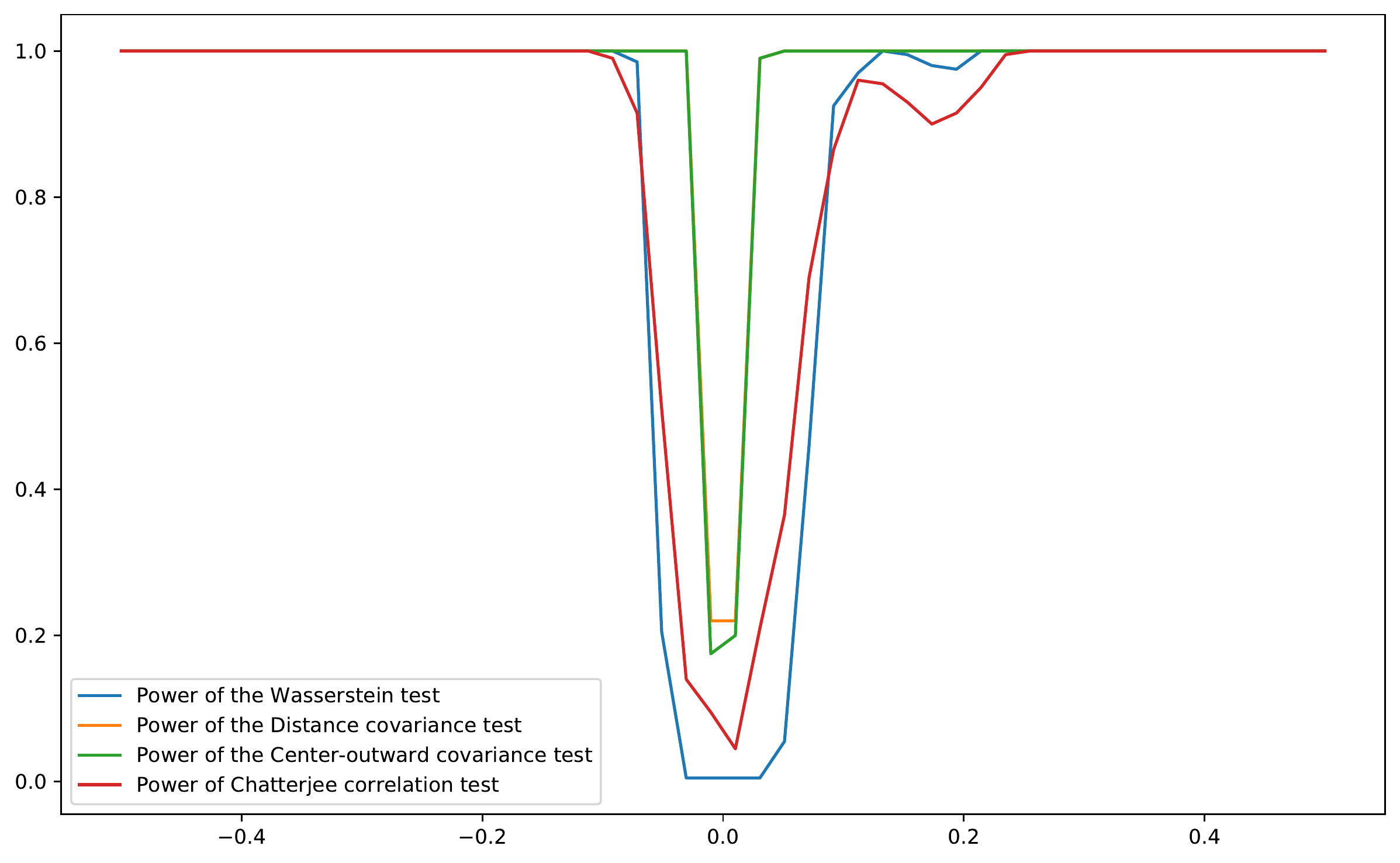}
\end{minipage}
\caption{Comparison of power functions for different tests of independence: $X_2=\rho X_1 +\sqrt{1-\rho^2} U$ for an independent uniform random variable $U$ (left) and  $\log(X^2_1)$ vs $\log(X_2^2)$ (right) for $1500$ samples and $200$ draws.}\label{fig:5}
\end{center}
\end{figure}
In Figure \ref{fig:5} we again plot the corresponding power functions for the uniform distribution on the unit interval $[0,1]$ with the correlated random variable $X_2=\rho X_1 +\sqrt{1-\rho^2} U$ for an independent uniform random variable $U$. We also consider the transform $\log(X^2_1)$ vs $\log(X_2^2)$.
In both cases, the tests discussed in \cite[Theorem 6]{szekely2007measuring} and \cite[Theorem 3.1]{shi2020distribution} seem to yield slightly superior behaviour of the corresponding power function, while our results as well as \cite[Theorem 2.1]{chatterjee2020new} exhibit a less spiked shape.

\section{Outlook}\label{sec:8}

While the theoretical properties of the Wassserstein correlation established in Sections \ref{sec:2} and \ref{sec:3} hold in great generality, the development of an estimation procecdure for $\overrightarrow{\mathcal{W}}$ on general Polish spaces is still an open problem. In this paper we have exhibited a viable estimation approach for compactly supported probability measures, leveraging results from \cite{backhoff2020estimating}. It is however a non-trivial task to extend their framework to general Polish spaces and show consistency of the adapted empirical measure. We leave this extension to future work.\\
In a similar vein, the test of independence derived in Section \ref{sec:4} relies on non-optimal convergence rates even for compactly supported probability measures $\pi$ (for a numerical example see Appendix \cite{jwb}[Section 2]). We believe that these can be improved once a deeper understanding of the exact distribution of $\overrightarrow{\mathcal{W}}_N$ in the asymptotic regime is available. We remark however that this a difficult task: obtaining convergence rates even for the (non-adapted) Wasserstein distance is still an active research field and recent works often rely on intricate probabilistic estimates (see e.g. \cite{fournier2015rate,sommerfeld2016inference}).
 
\section{Acknowledgements}
The author would like to thank the Bodhi Sen and Giovanni Puccetti for helpful discussions. Furthermore the author would like to thank the anonymous referees, an Associate Editor and the Editor for their constructive comments that improved the quality of this paper.

\bibliographystyle{plainnat}
\bibliography{bib}

\begin{appendix}
\section{Remaining proofs}

\begin{proof}[Proof of Theorem \ref{thm:easy}]
\begin{enumerate}[(i)]
\item Clearly $\overrightarrow{\mathcal{W}}(\pi)\ge 0$. Furthermore, replacing the $\mathcal{W}(\pi_{x_1},\nu)$-optimal coupling in the numerator of $\overrightarrow{\mathcal{W}}(\pi)$ by the the product coupling $\pi_{x_1}\otimes \nu$ we obtain the upper bound
\begin{align}\label{eq:upper}
\begin{split}
\int \mathcal{W}(\pi_{x_1},\nu)\,\mu(dx_1)&\le \int \int d(y,z) \,\pi_{x_1}(dy)\nu(dz)\mu(dx_1)\\
&=\int d(y,z)\,\nu(dy)\nu(dz).
\end{split}
\end{align}
\item If $\overrightarrow{\mathcal{W}}(\pi)=0$ then $\mathcal{W}(\pi_{x_1}, \nu)=0$ $\mu$-a.s. and thus $\pi_{x_1}=\nu$ $\mu$-a.s. by positive definiteness of the Wasserstein distance. In particular 
\begin{align}\label{eq:product}
\pi(A\times B) =\int_A \pi_{x_1}(B) \,\mu(dx_1) =\int_A \nu(B)\mu(dx_1)= (\mu\otimes\nu)(A\times B)
\end{align}
for any Borel subsets $A,B\subseteq \mathcal{X}$ and thus $\pi=\mu\otimes\nu$. On the other hand, if $\pi=\mu\otimes\nu$, then using again \eqref{eq:product} we conclude that $\pi_{x_1}=\nu$ $\mu$-a.s. by $\mu$-a.s. uniqueness of disintegrations. Thus $\mathcal{W}(\pi_{x_1}, \nu)=0$ $\mu$-a.s., which in turn implies $\overrightarrow{\mathcal{W}}(\pi)=0$. This shows the claim.
\item Note that cyclical monotonicity of optimal transport for the cost function $c(x,y)=d(x,y)$ (see e.g. \cite[Def. 5.1]{villani2008optimal}) implies that inequality \eqref{eq:upper} is strict unless $\pi_{x_1}=\delta_{f(x_1)}$ for some function $f:\mathcal{X}\to \mathcal{X}$: indeed, consider the product coupling $\pi_{x_1}\otimes \nu$ and define the set
\begin{align*}
A:= \left\{x_1\in \mathcal{X}:\ \exists\, y_1, \tilde{y}_1 \in \text{supp}(\pi_{x_1}),\ y_1\neq \tilde{y}_1\right\}.
\end{align*}
Let us assume towards a contradiction that $\mu(A)>0$. By the definition of the disintegration $x_1\mapsto \pi_{x_1}$ and tightness of probability measures we then obtain
\begin{align*}
\mu\left(\left\{ x_1\in \mathcal{X}:\ \exists\, y_1, \tilde{y}_1 \in \text{supp}(\pi_{x_1})\cap \text{supp}(\nu),\ y_1\neq \tilde{y}_1\right\}\right)>0.
\end{align*}
Next, by the definition of the product coupling $\pi_{x_1}\otimes \nu$ we have that
\begin{align*}
\mu( \left\{x_1\in \mathcal{X}:\ \exists (y_1,\tilde{y}_1), (\tilde{y}_1, y_1)\in \text{supp}(\pi_{x_1}\otimes \nu),\ y_1\neq \tilde{y}_1\right\})>0.
\end{align*}
Now we note that
\begin{align*}
d(y_1, \tilde{y}_1)+d(\tilde{y}_1, y_1)>d(y_1, y_1)+d(\tilde{y}_1, \tilde{y}_1)=0,
\end{align*}
so that $$\mu\left( \left\{ x_1\in \text{X}:\ \text{supp}(\pi_{x_1}\otimes \nu)\text{ is not cyclically monotone} \right\}\right)>0,$$ a contradiction. On the other hand for $\pi_{x_1}=\delta_{f(x_1)}$ we have
\begin{align*}
\int \mathcal{W}(\pi_{x_1},\nu)\,\mu(dx_1)&=\int d(f(x_1),z)\,\mu(dx_1)\,\nu(dz)\\
&=\int d(y,z)\,\nu(dy)\,\nu(dz).
\end{align*}
This concludes the proof.
\end{enumerate}
\end{proof}

\begin{proof}[Proof of Lemma \ref{lem:chatterjee}]
According to \cite[Theorem 1.1]{chatterjee2020new} we have
\begin{align*}
T^C(\pi)&= \frac{\int \text{Var}\left(\E\left[\mathds{1}_{\{Y \ge y\}} \mid X\right]\right)\,\nu(dy) }{\int \text{Var}\left( \mathds{1}_{\{Y\ge y\}}\right)\,\nu(dy)}\\
&= \frac{\int \int \left( \pi_{x_1}[y,\infty)-\int \pi_{x_1}[y,\infty)\,\mu(dx_1) \right)^2\,\mu(dx_1)\,\nu(dy)  }{\int \text{Var}\left( \mathds{1}_{\{Y\ge y\}}\right)\,\nu(dy)}\\
&= \frac{\int \int \left( \pi_{x_1}[y,\infty)-\nu[y,\infty) \right)^2\,\mu(dx_1)\,\nu(dy)  }{\int \text{Var}\left( \mathds{1}_{\{Y\ge y\}}\right)\,\nu(dy)}.
\end{align*}
In particular we now note that by Fubini's theorem
\begin{align*}
\int \int \left( \pi_{x_1}[y,\infty)-\nu[y,\infty) \right)^2\,\mu(dx_1)\,\nu(dy) &=\int \int \left( \pi_{x_1}[y,\infty)-\nu[y,\infty) \right)^2\,\nu(dy)\,\mu(dx_1)\\
&= \int \int \left( F_{\pi_{x_1}}(y)-F_{\nu}(y) \right)^2\,\nu(dy)\,\mu(dx_1)\\
&= \int \int_{[0,1]} \left( F_{\pi_{x_1}}(F^{-1}_\nu(y))-F_{\nu}(F^{-1}_\nu(y)) \right)^2\,dy\,\mu(dx_1)\\
&= \int \int_{[0,1]} \left( F_{\pi_{x_1}}(F^{-1}_\nu(y))-y \right)^2\,dy\,\mu(dx_1)\\
&\le \int \int_{[0,1]} \left| F_{\pi_{x_1}}(F_\nu^{-1}(y))-y \right| \,dy\,\mu(dx_1)\\
&=\int \mathcal{W}(\tilde{\pi}_{x_1}, \mathcal{U}([0,1]))\,\mu(dx_1)
\end{align*}
for $\tilde{\pi}_{x_1}:=(F_\nu)_{\#}\pi_{x_1}$ and $\mathcal{U}([0,1])$ is the uniform distribution on $[0,1]$.
On the other hand we also obtain from the above that
\begin{align*}
\left( \int \int \left( \pi_{x_1}[y,\infty)-\nu[y,\infty) \right)^2\,\mu(dx_1)\,\nu(dy) \right)^{1/2} &\ge \int \int_{[0,1]} \left| F_{\mu_{x_1}}(F_\nu^{-1}(y))-y \right| \,dy\,\mu(dx_1) \\
&= \int \mathcal{W}(\tilde{\pi}_{x_1}, \mathcal{U}([0,1]))\,\mu(dx_1).
\end{align*}
This concludes the proof.
\end{proof}

\begin{proof}[Proof of Lemma \ref{rem:1}]
Choosing $\gamma^{x_1}\in \Pi(\nu,\pi_{x_1})$ such that
\begin{align*}
\mathcal{W}(\nu,\pi_{x_1}) =\int |y-z|_2\,\gamma^{x_1}(dy, dz)
\end{align*}
for each $x_1\in \mathcal{X}$,
it is not hard to see that
\begin{align*}
T^{DGS}(\pi)&= \frac{ \int (|y-z|_2 -|\tilde{y}-\tilde{z}|)\,\gamma^{x_1}(dy,d\tilde{y})\,\gamma^{x_1}(dz, d\tilde{z}) \,\mu(dx_1)}{\int |y-z|_2 \,\nu(dy)\,\nu(dz)}\\
&\le \frac{ \int |y -z-(\tilde{y}-\tilde{z})|_2\,\gamma^{x_1}(dy,d\tilde{y})\,\gamma^{x_1}(dz, d\tilde{z}) \,\mu(dx_1)}{\int |y-z|_2 \,\nu(dy)\,\nu(dz)}\\
&\le \frac{ \int (|y-\tilde{y}|_2+|z -\tilde{z}|_2)\,\gamma^{x_1}(dy,d\tilde{y})\,\gamma^{x_1}(dz, d\tilde{z}) \,\mu(dx_1)}{\int |y-z|_2 \,\nu(dy)\,\nu(dz)}\\
&= \frac{ \int |y-\tilde{y}|_2\,\gamma^{x_1}(dy,d\tilde{y})\,\mu(dx_1)+\int |z -\tilde{z}|_2\,\gamma^{x_1}(dz, d\tilde{z}) \,\mu(dx_1)}{\int |y-z|_2 \,\nu(dy)\,\nu(dz)}\\
&= \frac{2\int \mathcal{W}(\pi_{x_1},\nu)\,\mu(dx_1)}{\int |y-z|_2 \,\nu(dy)\,\nu(dz)}=2\overrightarrow{\mathcal{W}}(\pi).
\end{align*}
In conclusion, in the case $(\mathcal{X},d)=(\R^d, |\cdot|_2)$, the functional $T^{DGS}(\pi)$ is dominated by $2\overrightarrow{\mathcal{W}}(\pi)$.
\end{proof}

\begin{proof}[Proof of Lemma \ref{lem:hell}]
Recall from \cite[Section 4]{geenens2020hellinger} that
\begin{align*}
T^H(\pi)=\int \int \left( \sqrt{ f_\pi(x_1, x_2)}-\sqrt{ f_\mu(x_1) f_\nu(x_2)} \right)^2\,dx_1\,dx_2.
\end{align*}
If $\mu,\nu$ have bounded support, then $\pi\in \Pi(\mu,\nu)$ also has bounded support and standard results on comparing metrics on $\text{Prob}(\R)$ (see e.g. \cite{gibbs2002choosing}) imply that there exists a constant $C>0$ such that
\begin{align*}
\mathcal{W}(\pi_{x_1},\nu)\le C \int  \left( \sqrt{ f_{\pi_{x_1}}(x_2)}-\sqrt{ f_\nu(x_2)} \right)^2\,dx_2.
\end{align*}
In particular 
\begin{align*}
\int \mathcal{W}(\pi_{x_1},\nu)f_{\mu}(x_1)\,dx_1 &\le C \int \int  \left( \sqrt{ f_{\pi_{x_1}}(x_2)}-\sqrt{ f_\nu(x_2)} \right)^2\,dx_2 \, f_{\mu}(x_1)\,dx_1\\
&\le C \int \int  \left( \sqrt{ f_{\mu}(x_1) f_{\pi_{x_1}}(x_2)}-\sqrt{ f_{\mu}(x_1) f_\nu(x_2)} \right)^2\,dx_2 \,dx_1\\
&=C \int \int  \left( \sqrt{ f_{\pi}(x_1,x_2)}-\sqrt{ f_{\mu}(x_1) f_\nu(x_2)} \right)^2\,dx_2 \,dx_1=CT^H(\pi)
\end{align*}
by Tonelli's theorem and the definition of the conditional density $ f_{\pi_{x_1}}(x_2)$.
\end{proof}

\begin{proof}[Proof of Lemma \ref{lem:pearson}]
Note that we can immediately read off the marginal distributions $\mu=\mathcal{N}(a_1, \sigma_1^2)$ and $\nu=\mathcal{N}(a_2, \sigma_2^2),$ as well as $$\pi_{x_1}=\mathcal{N}\left(a_2+\frac{\sigma_2}{\sigma_1} \rho (x_1-a_1), (1-\rho^2)\sigma_2^2\right).$$
Furthermore, by the explicit formula for the $2$-Wasserstein distance between Gaussians (see e.g. \cite[Simple example]{knott1984optimal}) one can compute
\begin{align*}
\mathcal{W}_2(\pi_{x_1}, \nu)^2&= \left(a_2+\frac{\sigma_2}{\sigma_1} \rho (x_1-a_1) -a_2\right)^2 +\sigma_2^2+(1-\rho^2)\sigma_2^2-2 \sqrt{(1-\rho^2)\sigma_2^4}\\
&=\left( \frac{\sigma_2}{\sigma_1} \rho (x_1-a_1) \right)^2 +\sigma_2^2+(1-\rho^2)\sigma_2^2-2\sigma^2_2 \sqrt{1-\rho^2},
\end{align*}
so that
\begin{align*}
\int \mathcal{W}_2(\pi_{x_1}, \nu)^2 \,\mu(dx_1)&= \rho^2\sigma_2^2 +\sigma_2^2+(1-\rho^2)\sigma_2^2-2\sigma_2^2 \sqrt{1-\rho^2}\\
&=2\sigma_2^2 \left(1-\sqrt{1-\rho^2}\right).
\end{align*}
Lastly
\begin{align*}
\int |y-z|^2 \,\nu(dy)\nu(dz)= 2\int |y|^2\,\nu(dy)-2\left(\int|z|\,\nu(dz)\right)^2=2\sigma_2^2
\end{align*}
and the claim follows.
\end{proof}

\begin{proof}[Proof of Theorem \ref{thm:triangle}]
Fix $\delta>0$ and take $\gamma\in \Pi(\mu,\tilde{\mu})$ such that
\begin{align}\label{eq:defn}
\int\left[ d(x_1,y_1)+ \mathcal{W}(\pi_{x_1}, \tilde{\pi}_{y_1})\right] \, \gamma(dx_1,dy_1)\le \mathcal{AW}(\pi,\tilde{\pi})+\delta.
\end{align}
A repeated application of the triangle inequality now yields
\begin{align*}
&\left| \int \mathcal{W}(\pi_{x_1}, \nu)\,\mu (dx_1) - \int \mathcal{W}(\tilde{\pi}_{y_1}, \tilde{\nu})\,\tilde{\mu}(dy_1) \right| \\
&=\left| \int \mathcal{W}(\pi_{x_1}, \nu)\,\gamma (dx_1,dy_1) - \int \mathcal{W}(\tilde{\pi}_{y_1}, \tilde{\nu})\,\gamma(dx_1,dy_1) \right| \\
&\le \int |\mathcal{W}(\pi_{x_1},\nu)-\mathcal{W}(\tilde{\pi}_{y_1}, \tilde{\nu})|\,\gamma(dx_1,dy_1)\\
&\le  \int \left[|\mathcal{W}(\pi_{x_1},\nu)-\mathcal{W}(\tilde{\pi}_{y_1},\nu)|+|\mathcal{W}(\tilde{\pi}_{y_1},\nu)-\mathcal{W}(\tilde{\pi}_{y_1}, \tilde{\nu})|\right]\,\gamma(dx_1,dy_1)\\
&\le  \int \left[\mathcal{W}(\pi_{x_1},\tilde{\pi}_{y_1})+\mathcal{W}(\nu, \tilde{\nu})\right]\,\gamma(dx_1,dy_1)\\
&\le \int \mathcal{W}(\pi_{x_1},\tilde{\pi}_{y_1})\, \gamma(dx_1,dy_1)+\int \mathcal{W} (\nu, \tilde{\nu})\,\gamma(dx_1,dy_1)\\
&\le  \mathcal{AW}(\pi, \tilde{\pi})+\delta+\mathcal{W}(\nu,\tilde{\nu}),
\end{align*}
where the last inequality follows from the particular choice of $\gamma$ in \eqref{eq:defn}.
As $\delta>0$ was arbitrary and noting that
\begin{align}\label{eq:ineq2}
\mathcal{W}(\nu,\tilde{\nu})\le \mathcal{W}(\pi,\tilde{\pi})\le \mathcal{AW}(\pi,\tilde{\pi}),
\end{align}
we conclude that
\begin{align*}
\left| \int \mathcal{W}(\pi_{x_1}, \nu)\,\mu(dx_1) - \int \mathcal{W}(\tilde{\pi}_{y_1}, \tilde{\nu})\,\tilde{\mu}(dy_1) \right|&\le \mathcal{AW}(\pi, \tilde{\pi})+\mathcal{W}(\nu,\tilde{\nu})\\
&\le 2\mathcal{AW}(\pi, \tilde{\pi}),
\end{align*}
which shows the first claim. Let us  define 
\begin{align*}
f(\nu,\tilde{\nu}):=\int  d(y,z)\,\nu(dy)\,\nu(dz) \cdot \int  d(y,z)\,\tilde{\nu}(dy)\,\tilde{\nu}(dz)
\end{align*}
and recall 
\begin{align*}
f(\tilde{\nu})= \int  d(y,z)\,\tilde{\nu}(dy)\,\tilde{\nu}(dz).
\end{align*}
The second claim now follows from writing
\begin{align*}
\left| \overrightarrow{\mathcal{W}}(\pi)-\overrightarrow{\mathcal{W}}(\tilde{\pi}) \right|&=\frac{1}{f(\nu,\tilde{\nu})} \Bigg| \int \mathcal{W}(\pi_{x_1}, \nu)\,\mu(dx_1) \int d(y,z)\,\tilde{\nu}(dy)\tilde{\nu}(dz) \\
&\qquad\qquad- \int \mathcal{W}(\tilde{\pi}_{y_1}, \tilde{\nu})\,\tilde{\mu}(dy_1)  \int d(y,z)\,\nu(dy)\,\nu(dz) \Bigg|\\
&\le \frac{1}{f(\nu,\tilde{\nu})} \Bigg[\Bigg|\int d(y,z)\,\tilde{\nu}(dy)\tilde{\nu}(dz)-\int d(y,z)\,\nu(dy)\,\nu(dz) \Bigg|\\
&\quad \cdot \int \mathcal{W}(\pi_{x_1}, \nu)\,\mu(dx_1)\\
&\quad + \Bigg| \int \mathcal{W}(\pi_{x_1}, \nu)\,\mu(dx_1) -\int \mathcal{W}(\tilde{\pi}_{y_1}, \tilde{\nu})\,\tilde{\mu}(dy_1) \Bigg|\\
&\quad \cdot \int d(y,z)\,\nu(dy)\,\nu(dz) \Bigg]\\
&\le \frac{1}{f(\tilde{\nu})}  \left[ g(\nu,\tilde{\nu})+\mathcal{AW}(\pi, \tilde{\pi})+\mathcal{W}(\nu,\tilde{\nu}) \right]\\
&\le \frac{1 }{f(\tilde{\nu})}  \left[g(\nu,\tilde{\nu})+ 2\mathcal{AW}(\pi, \tilde{\pi})\right]
\end{align*}
for any $x_0\in \mathcal{X}$, where we have used the estimate $\int \mathcal{W}(\pi_{x_1}, \nu)\,\mu(dx_1)\le \int d(y,z)\,\nu(dy)\,\nu(dz)$ and the definition of $f$ in the second inequality. Now let $\gamma\in \Pi(\tilde{\nu}, \nu)$ be an $\mathcal{W}$-optimal coupling between $\tilde{\nu}$ and $\nu$. Using again the triangle inequality we then conclude that
\begin{align*}
 g(\nu,\tilde{\nu})&= \left|\int d(y, z) \,\tilde{\nu}(dy)\, \tilde{\nu}(dz) - \int d(y, z)\, \nu(dy) \,\nu(dz)  \right|\\
 &\le \left|\int d(y, z) \,\tilde{\nu}(dy)\, \tilde{\nu}(dz) -\int d(y, \tilde{z})\, \tilde{\nu}(dy)\,\nu(d\tilde{z}) \right| \\
 &\quad +\left| \int d(\tilde{y}, z) \,\tilde{\nu}(d\tilde{y}) \,\nu(dz) - \int d(y, z)\, \nu(dy) \,\nu(dz)  \right| \\
 &\le \int   |d(y, z)-d(y, \tilde{z})| \,\tilde{\nu}(dy)\, \gamma(dz,d\tilde{z})\\
 &\quad +\int   |d(\tilde{y}, z)-d(y, z)| \,\nu(dz)\, \gamma(d\tilde{y},dy)\\
 &\le \int   d(z,\tilde{z}) \,\tilde{\nu}(dy)\, \gamma(dz,d\tilde{z})+ \int   d(\tilde{y},y) \,\tilde{\nu}(dz)\, \gamma(d\tilde{y},dy)\\
 &= 2 \mathcal{W}(\nu, \tilde{\nu})\le 2\mathcal{AW}(\pi, \tilde{\pi}).
\end{align*}

This concludes the proof.
\end{proof}

\begin{proof}[Proof of Lemma \ref{thm:almost.sure.convergence}]
The proof follows from the same arguments as in \cite[Proof of Theorem 1.3]{backhoff2020estimating} with a few minor changes. We first remark that it is enough to show the claim for $\pi$ with  continuous disintegration $x_1\mapsto \pi_{x_1}$. Indeed, the general case then follows exactly as in \cite[Proof of Theorem 1.3]{backhoff2020estimating}.\\
We now note that \cite[Proof of Lemma 3.4]{backhoff2020estimating} states explicitly that
\begin{align*}
E\left[\mathcal{W}\left(\pi_{1}, \widehat{\pi}_{1}^{N}\right)\right] \leq C R(N), 
\end{align*}
where the function $R$ is defined as
\begin{align*}
R:[0,+\infty) \rightarrow[0,+\infty], \quad R(u):=\left\{\begin{array}{ll}
u^{-1 / 2} & \text { if } d=1 \\
u^{-1 / 2} \log (u+3) & \text { if } d=2 \\
u^{-1 / d} & \text { if } d \geq 3.
\end{array}\right.
\end{align*}
Furthermore \cite[Proof of Lemma 3.4]{backhoff2020estimating} also states that
\begin{align*}
E\left[\sum_{G \in \Phi_{t}^{N}} \widehat{\pi}^{N}(G) \mathcal{W}\left(\mu_{G}, \widehat{\mu}_{G}^{N}\right) \mid \mathcal{G}_{t}^{N}\right]\leq R\left(\frac{N}{\left|\Phi^{N}\right|}\right),
\end{align*}
so that we can conclude as in \cite[Proof of Lemma 5.3]{backhoff2020estimating} that
\begin{align*}
\mathcal{A} \mathcal{W}\left(\pi, \picausal\right) \leq \delta+ C(\delta) \left( \Delta^N+ R\left(\frac{N}{\left|\Phi^{N}\right|}\right)\right)
\end{align*}
for all $N\in \N$ large enough,
where
\begin{align*}
\Delta^{N}&:=\sum_{f, N} \Delta_{ G}^{N}, \\
\Delta_{G}^{N}&:=\widehat{\mu}^{N}(G)\left(\mathcal{W}\left(\pi_{G}, \widehat{\pi}_{G}^{N}\right)-E\left[\mathcal{W}\left(\pi_{G},\widehat{\pi}_{G}^{N}\right) \mid \mathcal{G}_{t}^{N}\right]\right).
\end{align*}
We can now follow the arguments in  \cite[Proof of Theorem 5.3]{backhoff2020estimating}, noting that $$\lim_{N\to \infty}R\left(\frac{N}{\left|\Phi^{N}\right|}\right)=0$$ as $\lim_{N\to \infty} |\Phi^N|/N=0$ by assumption. This concludes the proof.
\end{proof}

\begin{proof}[Proof of Theorem \ref{thm:test1}]
We first bound the Wasserstein distance $\mathcal{W}(\picausal[G], \picausal[]_2)$ from above by quantities, whose distributions are easier to control. This goes back to a classical argument, see e.g. \cite[Lemma 5]{fournier2015rate}, or also \cite[Appendix A]{weed2017sharp} for a detailed discussion. In our specific case we can write

\begin{align}\label{eq:estimator}
\overrightarrow{\mathcal{W}}(\picausal)=\frac{\sum_{G\in \Phi^N} \frac{|n\in \{1, \dots, N\} \text{ s.t. } X_1^n\in G|}{N}\, \mathcal{W}( \picausal[G], \picausal[]_2) } {\frac{1}{N^2}\sum_{n,m=1}^N |\varphi^N(X_2^n)-\varphi^N(X_2^m)|}
\end{align}

and use the fact that both  $\picausal[G]$ and $\picausal[]_2$ are finitely supported on $\varphi^N([0,1]^d)$.  Together with the observation that $\text{diam}([0,1]^d)=\sqrt{d}$ we can thus bound the Wasserstein distance in \eqref{eq:estimator} from above as follows:

\begin{align*}
\mathcal{W}( \picausal[G], \picausal[]_2)&\le \sqrt{d} \sum_{H\in\Phi^N} |\picausal[G](H)- \picausal[]_2(H)|\\
&\le  \sqrt{d} \sum_{H\in\Phi^N} \Big|\frac{|n\in \{1,\dots, N\} \text{ s.t. } X_1^n\in G, X_2^n\in H|}{|n\in \{1, \dots, N\} \text{ s.t. } X_1^n\in G|}\\
&\qquad\qquad\qquad\qquad\qquad- \frac{|n\in \{1,\dots, N\} \text{ s.t. } X_2^n\in H|}{N}\Big|,
\end{align*}

so that

\begin{align*}
&\sum_{G\in \Phi^N} \frac{|n\in \{1, \dots, N\} \text{ s.t. } X_1^n\in G|}{N}\, \mathcal{W}( \picausal[G], \picausal[]_2)\\
&\le 
\sqrt{d}
\sum_{G\in \Phi^N}\frac{|n\in \{1, \dots, N\} \text{ s.t. } X_1^n\in G|}{N}  \\
&\quad\cdot  \sum_{H\in\Phi^N} \left|\frac{|n\in \{1,\dots, N\} \text{ s.t. } X_1^n\in G, X_2^n\in H|}{|n\in \{1, \dots, N\} \text{ s.t. } X_1^n\in G|}- \frac{|n\in \{1,\dots, N\} \text{ s.t. } X_2^n\in H|}{N}\right|\\
&=\sqrt{d}
\sum_{G\in \Phi^N} \sum_{H\in\Phi^N} \Bigg|\frac{|n\in \{1,\dots, N\} \text{ s.t. } X_1^n\in G, X_2^n\in H|}{N}\\
&\qquad\qquad\qquad\qquad- \frac{|n\in \{1, \dots, N\} \text{ s.t. } X_1^n\in G|}{N}\cdot \frac{|n\in \{1,\dots, N\} \text{ s.t. } X_2^n\in H|}{N}\Bigg|\\
&=:\sqrt{d}\,\tilde{T}_N(\pi).
\end{align*}

Up to the constant $\sqrt{d}$ the term $\tilde{T}_N(\pi)$ is a classical non-parametric estimator for independence of $\mu$ and $\nu$, see e.g. \cite{gretton2010consistent}.
More precisely \cite[Theorem 1]{gretton2010consistent} states that under the assumption $\pi=\mu\otimes \nu$ one has

\begin{align}\label{eq:estimate}
\begin{split}
\P( \tilde{T}_N(\pi) \ge \epsilon)&\le 2^{(|\Phi^N|+1)^2} \exp\left(-\frac{N \varepsilon^{2}}{ 2}\right)\\
&=\exp\left(  \log(2) (|\Phi^N|+1)^2-\frac{\varepsilon^{2}N}{ 2} \right)
\end{split}
\end{align}
for any $\epsilon>0$. We thus conclude that
\begin{align*}
\begin{split}
P\left(\int \mathcal{W}\left(\picausal[x_1],\picausal[]_2\right)\,\picausal[]_1(dx_1) \ge \epsilon \right)&\le 
P( \sqrt{d}\, \tilde{T}_N(\pi) \ge \epsilon)\\
&\le \exp\left(  \log(2) (|\Phi^N|+1)^2-\frac{\varepsilon^{2}N}{2d}\right).
\end{split}
\end{align*}
This shows the first claim.

In particular choosing $\epsilon= 2\sqrt{d\log(2)}\,(|\Phi^N|+1)/\sqrt{N}$ in \eqref{eq:estimate} yields
\begin{align*}
P\left( \tilde{T}_N(\pi) \ge 2\sqrt{d\log(2)}\, \frac{|\Phi^N|+1}{\sqrt{N}}\right)\le  \exp(-|\Phi^N|^2), 
\end{align*}
which is summable by assumption. Thus a Borel-Cantelli argument implies that
\begin{align*}
\tilde{T}_N(\pi) \le 2\sqrt{d\log(2)}\, \frac{|\Phi^N|+1}{\sqrt{N}}
\end{align*}
for all sufficiently large $N\in \N$. Lastly, noting that by the law of large numbers
\begin{align*}
\lim_{N\to\infty} \int |y-z|\,\picausal[]_2(dy)\,\picausal[]_2(dz)= \int |y-z|\,\nu(dy)\,\nu(dz),
\end{align*}
where the term on the right is positive by assumption, we can choose an appropriate constant $C(\nu)>0$ in order to obtain
\begin{align*}
\overrightarrow{\mathcal{W}}(\picausal) \le C(\nu)\frac{|\Phi^N|}{\sqrt{N}}
\end{align*}
for all $N\in \N$ sufficiently large. On the other hand, if $\pi\neq \mu\otimes\nu$, then $\mathcal{AW}$-consistency of $\picausal$ implies that $$\int \mathcal{W}( \picausal[G], \picausal[]_2) \,\picausal[]_1(dx_1)$$ does not converge to zero, so that there exists $\delta>0$ such that $\overrightarrow{\mathcal{W}}(\picausal)\ge \delta$ for all $N\in \N$ sufficiently large.
\end{proof}

Let us next recall the following lemma, which is used in the proof of Corollary \ref{cor:go}.
\begin{lemma}[{\cite[Theorem 3]{gretton2010consistent}}]\label{prop:1}
Under the assumption that $\mu$ and $\nu$ are non-atomic and $\pi=\mu\otimes\nu$, there exists a centering sequence $$C_N= C_N(\mu,\nu)\le\sqrt{\frac{2}{\pi}}\frac{|\Phi^N|}{\sqrt{N}} $$ such that
\begin{align*}
\sqrt{N} (\tilde{T}_N(\pi)-C_N)/\sigma \Rightarrow \mathcal{N}(0,1),
\end{align*}
where $\sigma=1-2/\pi$ and
\begin{align*}
\tilde{T}_N(\pi)&:=
\sum_{G\in \Phi^N} \sum_{H\in\Phi^N} \Bigg|\frac{|n\in \{1,\dots, N\} \text{ s.t. } X_1^n\in G, X_2^n\in H|}{N}\\
&\qquad\qquad- \frac{|n\in \{1, \dots, N\} \text{ s.t. } X_1^n\in G|}{N}\cdot \frac{|n\in \{1,\dots, N\} \text{ s.t. } X_2^n\in H|}{N}\Bigg|.
\end{align*}
\end{lemma}

While the estimate of $|f((\hat{\pi}^N)^2)-f(\nu)|$ in terms of $\mathcal{W}(\pi, \hat{\pi}^N)$ is useful for the proof of Corollary \ref{cor:consistency}, the following result provides sharper convergence rates for the case $\hat{\pi}^N=\picausal$:

\begin{lemma}\label{lem:O}
We have
$$\left(\sqrt{N}\wedge \frac{1}{2\cdot \sup_{x} |\varphi^N(x)-x|} \right)\left(f(\picausal[]_2)-f(\nu)\right)=O_P(1).$$
\end{lemma}

\begin{proof}
We note that
\begin{align*}
f(\picausal[]_2)-f(\nu)&= \int |y-z|_2\, \picausal[]_2(dy)\, \picausal[]_2(dz) -\int |y-z|_2\,\nu(dy)\,\nu(dz)\\
&=\frac{1}{N^2} \sum_{i,j=1}^N |\varphi^N(X_2^i)- \varphi^N(X_2^j)| -\int |y-z|_2\,\nu(dy)\,\nu(dz)\\
&=\frac{1}{N^2} \sum_{i,j=1, \ i\neq j}^N |\varphi^N(X_2^i)- \varphi^N(X_2^j)| -\int |y-z|_2\,\nu(dy)\,\nu(z)\\
&\le 2\cdot \sup_{x} |\varphi^N(x)-x|+\frac{1}{N^2} \sum_{i,j=1, \ i\neq j}^N |X_2^i- X_2^j| -\int |y-z|_2\,\nu(dy)\,\nu(dz).
\end{align*}
Using the CLT for U-statistics for the kernel $h(y_1, y_2)=|y_1-y_2|$ (see e.g. \cite[Theorem 12.3]{van1998asymptotic}) we conclude that 
$$\sqrt{N}\left(\frac{1}{N^2} \sum_{i,j=1, \ i\neq j}^N |X_2^i- X_2^j| -\int |x_2-y|_2\,\nu(dx_2)\,\nu(dy)\right)=O_P(1),$$
which shows the claim.
\end{proof}
In the following sections we discuss convergence rates of the estimator $\overrightarrow{\mathcal{W}}(\picausal)$, first for the independent case $\pi=\mu\otimes\nu$ and subsequently for the general case.

Lastly, for the proofs in Section \ref{sec:5} we state here some of the main results from \cite{backhoff2020estimating} for the convience of the reader:

\begin{lemma}[Average rate of $\mathcal{AW}(\pi, \picausal)$, see {\cite[Theorem 1.5]{backhoff2020estimating}}]
\label{thm:rates.unit.cube}
	Under Assumption \ref{ass:lipschitz.kernel}, there is a constant $C>0$ such that
	\begin{align}
	\label{eq:mean.speed.rate}
	\begin{split}
	E\Big[ \mathcal{AW}(\mu,\picausal)\Big]
	&\leq C \cdot
	\begin{cases}
	N^{-1/3} &\text{for } d=1,\\
	N^{-1/4}\log(N+1) &\text{for } d=2,\\
	N^{-1/(2d)} &\text{for } d\ge 3,
	\end{cases} \\
	&=:C \cdot	\mathop{\mathrm{rate}}(N) 
	\end{split}
	\end{align}
	for all $N\geq 1$.
\end{lemma}

In the theorem above, the constant $C$ depends on $d$ and the Lipschitz-constant in Assumption \ref{ass:lipschitz.kernel}. Furthermore, \cite{{backhoff2020estimating}} also show the following concentration inequality:

\begin{lemma}[Deviation of $\mathcal{AW}(\pi, \picausal)$, see {\cite[Theorem 1.7]{backhoff2020estimating}}]
\label{thm:deviation}
	Under Assumption \ref{ass:lipschitz.kernel}, there are constants $c,C>0$ such that 
	\[ P\Big[ \mathcal{AW}(\mu,\picausal)  \geq C\mathop{\mathrm{rate}}(N)+\varepsilon \Big]
	\leq 4\exp\Big( -cN\varepsilon^2 \Big) \]
	for all $N\geq 1$ and all $\varepsilon>0$.
\end{lemma}
As above, the constants $c,C$ depend on $d$, and the Lipschitz constant in Assumption \ref{ass:lipschitz.kernel}. \\

\section{Remaining numerical examples}\label{sec:app_b}

Lastly we consider the distribution of 
\begin{align}\label{eq:hist}
\frac{\sqrt{N}}{\sigma}\left(\int \mathcal{W}\left(\picausal[x_1],\picausal[]_2\right)\,\picausal[]_1(dx_1)-C_N(\nu)\right),
\end{align}
which we bounded stochastically by a standard normal distribution in Corollary \ref{cor:go}. The empirical study in Figure \ref{fig:hist} shows that even if the constant $C_N(\mu,\nu)$ from Lemma \ref{prop:1} is chosen such that the distribution is approximately centralised, \eqref{eq:hist} seems to possess slimmer tails than the standard normal distribution. In conclusion it seems unlikely that a CLT holds.

\begin{figure}[h!]
\begin{center}
\begin{minipage}{0.48\textwidth}
\includegraphics[scale=0.27]{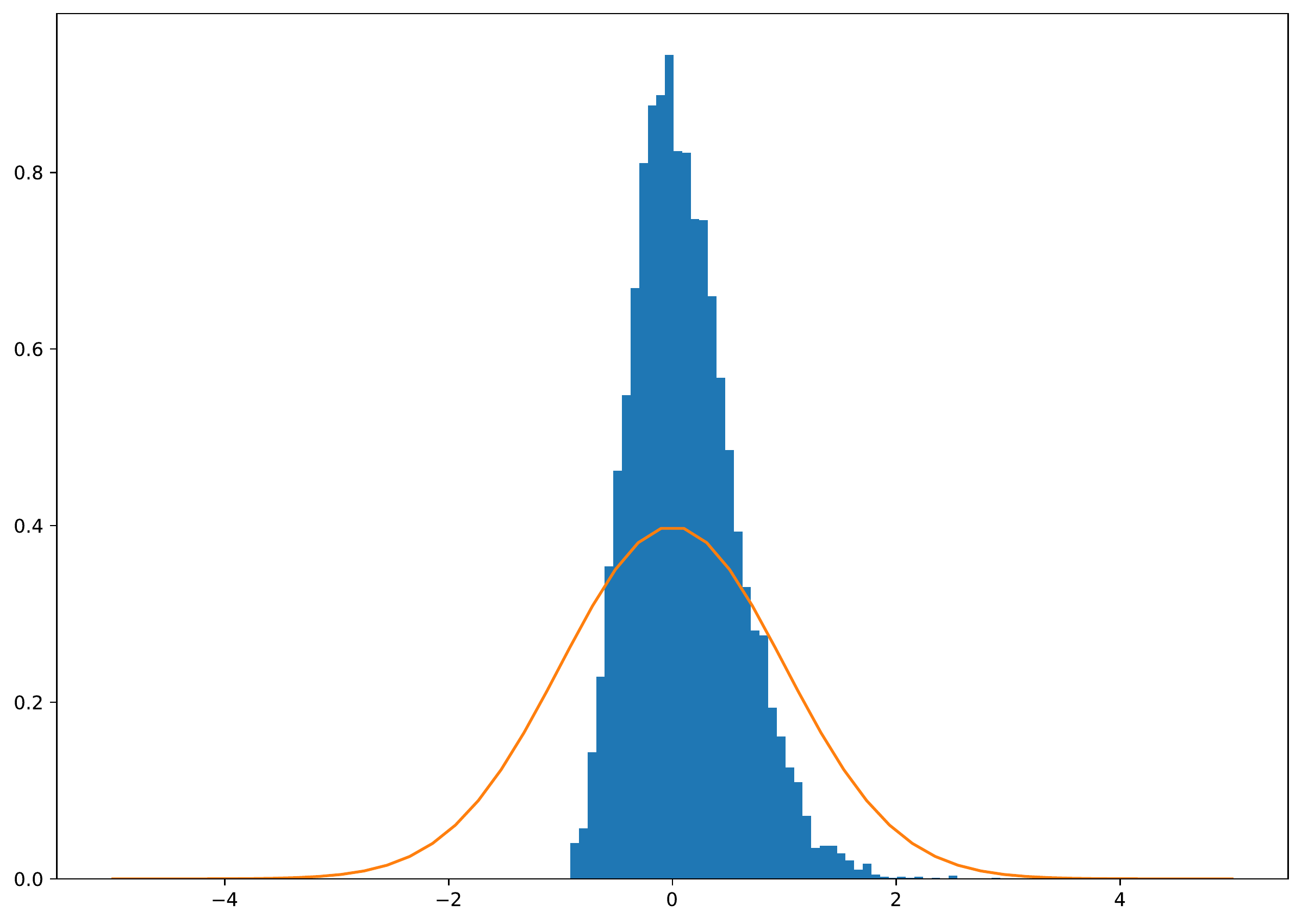}
\end{minipage}
\quad
\begin{minipage}{0.48\textwidth}
\includegraphics[scale=0.43]{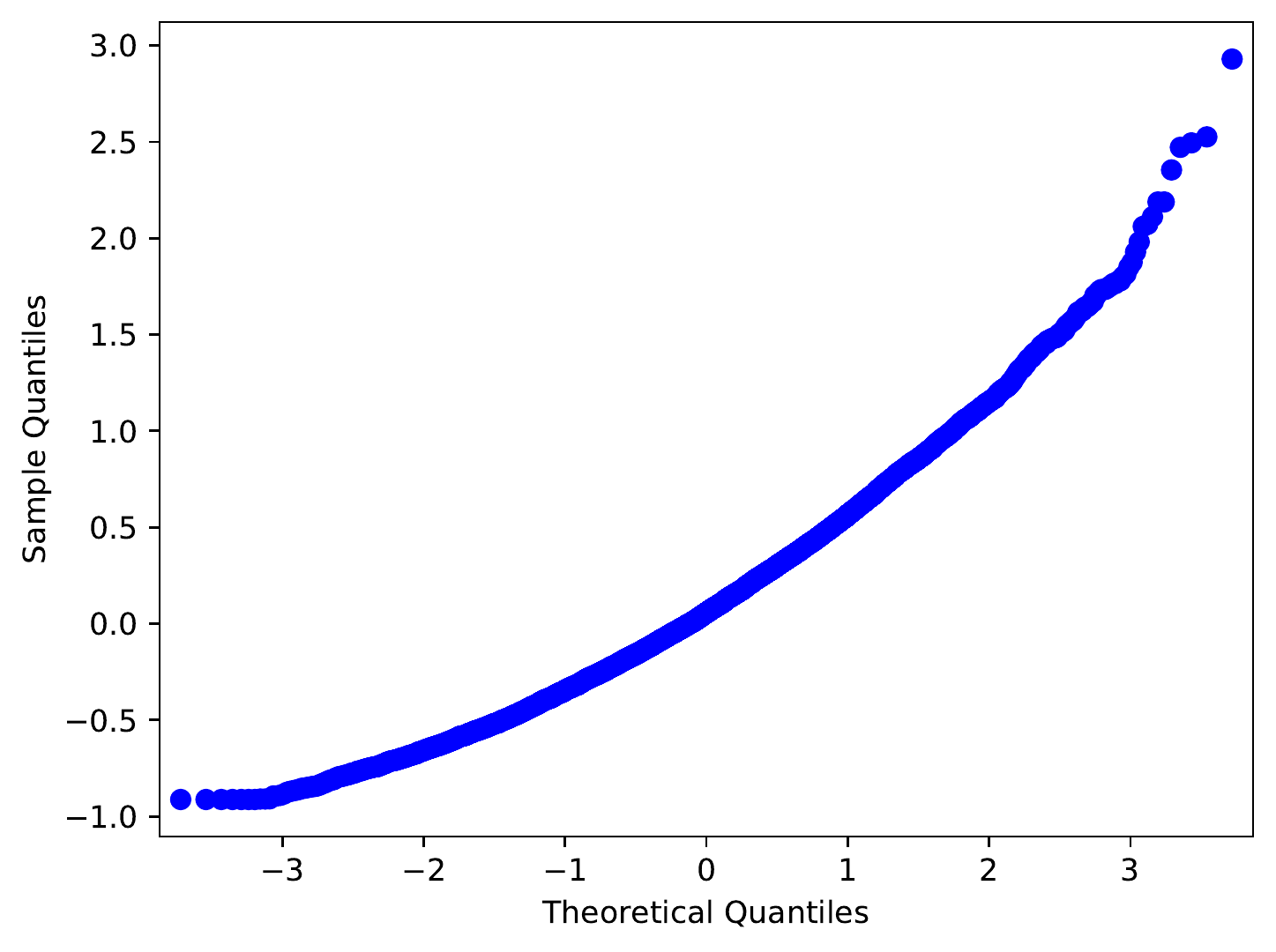}
\end{minipage}
\caption{Comparison of the (shifted) histogram of \eqref{eq:hist} for 500 samples and $10^4$ draws with a standard normal density function.}\label{fig:hist}
\end{center}
\end{figure}
\end{appendix}

\end{document}